\documentclass[a4paper, 12pt, reqno, english]{amsart}
\usepackage[utf8]{inputenc}
\usepackage[T1]{fontenc}
\usepackage{babel}

\usepackage[bookmarks = true, colorlinks, citecolor = blue, urlcolor = blue]{hyperref}
\usepackage[capitalize, nameinlink]{cleveref}
\usepackage[titletoc, toc, title]{appendix}

\hypersetup{pdfstartview=}

\usepackage{amssymb}
\usepackage{tikz}
\usepackage{tikz-cd}

\numberwithin{equation}{section}
\numberwithin{figure}{section}

\theoremstyle{plain}
\newtheorem{theorem}{Theorem}[section]

\theoremstyle{plain}
\newtheorem*{theorem*}{Theorem}

\theoremstyle{plain}
\newtheorem{proposition}[theorem]{Proposition}

\theoremstyle{plain}
\newtheorem{lemma}[theorem]{Lemma}

\theoremstyle{plain}
\newtheorem{corollary}[theorem]{Corollary}

\theoremstyle{definition}
\newtheorem{definition}[theorem]{Definition}

\theoremstyle{definition}

\theoremstyle{definition}

\theoremstyle{remark}
\newtheorem{remark}[theorem]{Remark}

\addto\extrasenglish{
  
}

\newcommand{\Uqg}{U_q(\mathfrak{g})}

\newcommand{\CqG}{\mathbb{C}_q[G]}

\newcommand{\id}{\mathrm{id}}

\newcommand{\bbC}{\mathbb{C}}

\newcommand{\wt}{\mathrm{wt}}

\newcommand{\diff}{\mathrm{d}}
\newcommand{\del}{\partial}
\newcommand{\delbar}{\bar{\partial}}

\newcommand{\EV}{\mathsf{E}}
\newcommand{\EVp}{\mathsf{E}^\prime}
\newcommand{\CV}{\mathsf{C}}
\newcommand{\CVp}{\mathsf{C}^\prime}

\newcommand{\braid}{\hat{\mathsf{R}}}

\newcommand{\met}{g}
\newcommand{\metMP}{g_{- +}}
\newcommand{\metPM}{g_{+ -}}

\newcommand{\calc}{\Omega}
\newcommand{\calcM}{\Omega_-}
\newcommand{\calcP}{\Omega_+}

\newcommand{\subalg}{\mathcal{B}}

\newcommand{\conn}{\nabla}

\newcommand{\opS}{\mathsf{S}}
\newcommand{\opSt}{\tilde{\mathsf{S}}}
\newcommand{\opT}{\mathsf{T}}

\newcommand{\lieg}{\mathfrak{g}}

\newcommand{\sigmaMM}{\sigma_{- -}}
\newcommand{\sigmaMP}{\sigma_{- +}}
\newcommand{\sigmaPM}{\sigma_{+ -}}
\newcommand{\sigmaPP}{\sigma_{+ +}}

\newcommand{\splitmap}{\boldsymbol{s}}
\newcommand{\curv}{R_\nabla}
\newcommand{\ricci}{\mathrm{Ricci}}
\newcommand{\ricciMP}{\mathrm{Ricci}_{- +}}
\newcommand{\ricciPM}{\mathrm{Ricci}_{+ -}}
\newcommand{\scal}{\mathrm{scal}}

\newcommand{\talg}{T_\subalg}

\newcommand{\imet}{(\cdot, \cdot)}
\newcommand{\qdim}{\mathrm{qdim}}
\newcommand{\Trq}{\mathrm{Tr}_q}

\newcommand{\cPM}{c_{+ -}}
\newcommand{\cMP}{c_{- +}}

\begin{document}

\title[Quantum Riemannian geometry of quantum projective spaces]{Quantum Riemannian geometry\\ of quantum projective spaces}

\author{Marco Matassa}

\address{OsloMet – Oslo Metropolitan University}

\email{marco.matassa@oslomet.no}

\begin{abstract}
We study the quantum Riemannian geometry of quantum projective spaces of any dimension.
In particular, we compute the Riemann and Ricci tensors using previously introduced quantum metrics and quantum Levi-Civita connections.
We show that the Riemann tensor is a bimodule map and derive various consequences of this fact.
We prove that the Ricci tensor is proportional to the quantum metric, giving a quantum analogue of the Einstein condition, and compute the corresponding scalar curvature.
\end{abstract}

\maketitle

\section*{Introduction}

\subsection*{Background}
Extending the ideas of classical geometry to quantum spaces is an active area of research, with many different approaches having been proposed over the years.
The underlying idea is that certain non-commutative algebras should be regarded as consisting of functions on (virtual) \emph{non-commutative spaces}, usually arising from classical spaces by some quantization procedure.
A rich source of examples is given by the theory of \emph{compact quantum groups} (see for instance \cite{klsc}), which allows to quantize any compact semisimple Lie group, as well as large classes of homogeneous spaces, among which we have the \emph{quantum projective spaces}.
The aim of this paper is to study the \emph{quantum Riemannian geometry} of these quantum homogeneous spaces, continuing the investigation initiated in \cite{paper-projective}.
The results obtained here generalize those obtained by Majid for the quantum 2-sphere in \cite{majid-sphere} (and later revisited with Beggs in \cite{bm11, bm18, quantum-book}).

First we should specify what we mean by quantum Riemannian geometry, since the framework one uses can depend strongly on the approach considered.
For instance, a popular approach to non-commutative geometry is the one introduced by Connes \cite{connes}, which studies a quantum space by representing its algebra on a Hilbert space and using a densely-defined operator to extract the geometrical information.
This leads to the concept of spectral triples, which for the quantum projective spaces have been widely studied, see for instance \cite{dirac-sphere, dirac-plane, dirac-projective}.
On the other hand, the approach we consider in this paper is more algebraic and a thorough discussion of it can be found in the book \cite{quantum-book}.
It has the advantage of being a more straightforward adaptation of the classical picture, which lets us easily introduce a quantum analogue of the metric tensor, study the compatibility of a connection with a metric, and define geometrical objects like the Ricci tensor.

We summarize the main ingredients of this approach, which we refer to as \emph{quantum Riemannian geometry}, following \cite{quantum-book}.
Given a non-commutative algebra $A$, a convenient starting point for a geometrical study is the introduction of a \emph{differential calculus} $\Omega^\bullet$ over $A$.
This choice is not unique in general, but for many examples there are canonical choices.
The differential calculus enters the picture mainly through its first and second-order parts, denoted by $\Omega^1$ and $\Omega^2$ respectively.
A \emph{quantum metric} can be defined as an element $g \in \Omega^1 \otimes_A \Omega^1$ admitting an appropriate \emph{inverse metric} $\imet: \Omega^1 \otimes_A \Omega^1 \to A$.
Next, a connection on $\Omega^1$ is a map $\nabla: \Omega^1 \to \Omega^1 \otimes_A \Omega^1$ satisfying an appropriate Leibniz rule.
This is a standard notion in the algebraic context, as well as that of torsion-free connection.
Furthermore, a notion of \emph{compatibility} between a connection $\nabla$ and a quantum metric $g$ can be formulated.
Requiring torsion-freeness and compatibility with the metric leads to the notion of \emph{quantum Levi-Civita connection}.
The existence of these objects is not guaranteed, at this level of generality.

Further geometrical aspects can be introduced, provided we have a quantum metric and a connection.
We define the \emph{Riemann tensor} $R_\nabla: \Omega^1 \to \Omega^2 \otimes_A \Omega^1$ as the curvature of the connection $\nabla$, which is a standard notion.
Next, the \emph{Ricci tensor} can be defined as a certain element $\ricci \in \Omega^1 \otimes_A \Omega^1$, which is constructed via an appropriate contraction of the Riemann tensor using the quantum metric and its inverse.
This definition also requires an auxiliary choice, namely a map $\splitmap: \Omega^2 \to \Omega^1 \otimes_A \Omega^1$ splitting the wedge product of $\Omega^\bullet$.
Also the \emph{scalar curvature} $\scal \in A$ can be defined as a further contraction of the Ricci tensor.
Finally, a quantum analogue of the \emph{Einstein condition} can be formulated naturally, as the requirement that the Ricci tensor should be proportional to the quantum metric $g$.

\subsection*{Quantum projective spaces}
We consider the framework discussed above in the case of quantum projective spaces.
Starting from the quantum special linear group $\bbC_q[SL_{r + 1}]$, we can define the algebra corresponding to the quantum projective space of complex dimension $r$, which we denote by $\subalg$.
This algebra admits a \emph{canonical} differential calculus $\calc^\bullet$, introduced by Heckenberger-Kolb in \cite{locally-finite, heko} (more generally for any \emph{quantum irreducible flag manifold}).
We denote its first-order part simply by $\Omega$.
This opens the door to the study of the quantum Riemannian geometry of $\subalg$.
The case $r = 1$, also known as the \emph{Podleś sphere}, was originally studied by Majid in \cite{majid-sphere} (and later expanded, as mentioned above).

Our aim in this paper is to investigate the geometry of general quantum projective spaces, that is for any value of $r$.
This study was initiated in \cite{paper-projective}, where the main results were the introduction of a particular quantum metric $\met \in \calc \otimes_\subalg \calc$ and a quantum Levi-Civita connection $\conn: \calc \to \calc \otimes_\subalg \calc$.
In the classical limit, the metric $\met$ reduces to the Fubini-Study metric and the connection $\conn$ reduces to the corresponding Levi-Civita connection (acting on the cotangent bundle).
These are the natural ingredients to use for the study of further geometrical aspects of quantum projective spaces, like the Riemann and Ricci tensors.

\subsection*{Results}
Let us now summarize the results of this paper.
We start our investigation by studying in some detail the \emph{generalized braiding} $\sigma: \calc \otimes_\subalg \calc \to \calc \otimes_\subalg \calc$ corresponding to the quantum Levi-Civita connection.
This map, which classically amounts to the flip map, is uniquely defined by the fact that $\conn$ is a \emph{bimodule connection}, as shown in \cite{paper-projective}.
First we prove that the components of $\sigma$, relative to the decomposition $\calc = \calcP \oplus \calcM$ into "holomorphic and antiholomorphic parts", satisfy appropriate quadratic relations.
Then we use $\sigma$ to give a new presentation of the differential calculus $\calc^\bullet$, paralleling the classical picture.

\begin{theorem*}
The differential calculus $\calc^\bullet$ can be presented as
\[
\calc^\bullet = \talg(\calc) / \langle \mathrm{im}(\id + \sigma) \rangle,
\]
where $\talg(\calc)$ is the tensor algebra of the $\subalg$-bimodule $\calc$.
\end{theorem*}

In particular this implies the property $\wedge \circ (\id + \sigma) = 0$, which can be used to relate two different formulations of the notion of metric compatibility.

Before getting into the more geometrical aspects, we prove various results related to the differential calculus $\calc^\bullet$ and the inverse metric $\imet$.
We introduce a one-parameter family of maps $\splitmap: \calc^2 \to \calc \otimes_\subalg \calc$ \emph{splitting} the wedge product, that is satisfying $\splitmap \circ \wedge = \id$, which we are going to use to define the Ricci tensor.
We compute the \emph{quantum metric dimension} of $\calc^\bullet$, which is what we call the composition $\imet \circ \met$, corresponding in the classical case to the dimension of the space in consideration. We show that it is a scalar and determine its value using the notion of quantum dimension.
Finally we show that the inverse metric $\imet$ satisfies a certain \emph{twisted symmetry}, which involves the generalized braiding $\sigma$.

We now come to the Riemann tensor $\curv$ corresponding to the quantum Levi-Civita connection $\conn$.
We explicitly compute its expression on the generators of $\calc$.
These computations show that $\curv \in \calc^{(1, 1)} \otimes \calc$, where $\calc^{(1, 1)} \subset \calc^2$ denotes the space of $(1, 1)$-forms with respect to the naturally defined complex structure.
Moreover we prove the following result.

\begin{theorem*}
We have that $\curv: \calc \to \calc^2 \otimes_\subalg \calc$ is a $\subalg$-bimodule map.
\end{theorem*}

We remark that in general the curvature is a left module map, but is not guaranteed to be a right module map.
This non-trivial result has various interesting applications, which we discuss: it implies that the generalized braiding $\sigma$ can be extended to a map $\calc \otimes_\subalg \calc^n \to \calc^n \otimes_\subalg \calc$ for any $n \in \mathbb{N}$ in a natural way; it gives a certain antisymmetry property for the Riemann tensor; it implies that $\sigma$ satisfies the braid equation, possibly up to symmetric terms.

Finally we come to the Ricci tensor, whose definition depends on the auxiliary choice of a splitting map.
We compute its expression on the generators of $\calc$, using the one-parameter family of splitting maps $\splitmap$ discussed above.
We show that this free parameter can be uniquely fixed by requiring that $\ricci$ is symmetric. This leads to the following result.

\begin{theorem*}
We have a quantum analogue of the Einstein condition, that is
\[
\ricci = k \met.
\]
\end{theorem*}

The Einstein constant $k$ is explicitly computed, and reduces to the corresponding value for the Fubini-Study metric in the classical limit.
We also compute the scalar curvature, which is a constant proportional to $k$ that also reduces to the correct classical value.

\subsection*{Organization}
The paper is organized as follows.
In \cref{sec:preliminaries} we discuss various preliminaries on quantum groups, focusing in particular on some properties of their categories of finite-dimensional modules.
In \cref{sec:projective-spaces} we summarize the various results obtained in \cite{paper-projective} for the quantum projective spaces, including quantum metrics and quantum Levi-Civita connections. This section also contains the various relations that form the technical backbone of this paper.
In \cref{sec:generalized-braiding} we prove various properties satisfied by the generalized braiding $\sigma$, in particular how it can be used to give a new presentation of the differential calculus $\calc^\bullet$.
In \cref{sec:additional-results} we prove various result involving the differential calculus and the inverse metric $\imet$, as discussed above.
In \cref{sec:riemann-tensor} we compute the Riemann tensor, show that it is a bimodule map and derive various consequences of this fact.
Finally, in \cref{sec:ricci-tensor} we use these results to compute the Ricci tensor. We show that, upon choosing an appropriate splitting map, we can have the Ricci tensor to be symmetric. Moreover this implies the quantum Einstein condition discussed above. We also compute the scalar curvature in this case.

This paper also contains various appendices, to which we relegate some technical matters.
In \cref{sec:classical-formulae} we write the classical limits of the various expressions obtained here, for better comparison between the classical and the quantum case.
In \cref{sec:properties-S} we recall various properties satisfied by the maps $\opS$ and $\opSt$ introduced in \cite{paper-projective}, which we use extensively in our computations here.
In \cref{sec:some-identities} we prove various identities that are used in the main text.
Finally in \cref{sec:relations-degree-two} we derive various equivalent expressions for the mixed relations between the generators of $\calc$, holding in the degree-two part of the differential calculus.

\medskip
\textbf{Acknowledgements.} I would like to thank Jyotishman Bhowmick for a discussion related to certain aspects of differential calculi.

\section{Preliminaries}
\label{sec:preliminaries}

In this section we recall some essential facts about quantum groups, in particular about the quantized enveloping algebras $\Uqg$.
We focus on the category of finite-dimensional $\Uqg$-modules, which is a braided monoidal category with duality, recalling various details about the braiding and the duality morphisms.
In what follows we use the language of tensor categories to make our computations more transparent, with \cite{egno} as our main reference.

\subsection{Quantum groups}

Let $\lieg$ be a complex simple Lie algebra.
We denote by $(\cdot, \cdot)$ the non-degenerate symmetric bilinear form induced by the Killing form, normalized in such a way that we have $(\alpha, \alpha) = 2$ for all short roots $\alpha$.
Write $\{\alpha_i\}_{i = 1}^r$ for the simple roots and $\{\omega_i\}_{i = 1}^r$ for the fundamental weights.
We also write, as customary, $\rho$ for the half-sum of the positive roots, or equivalently for the sum of all fundamental weights.

Given a real number $q$ such that $0 < q < 1$, the \emph{quantized enveloping algebra} $\Uqg$ is defined as a certain Hopf algebra deformation of the enveloping algebra $U(\lieg)$.
It has generators $\{ K_i, \ E_i, \ F_i \}_{i = 1}^r$ with $r = \mathrm{rank}(\lieg)$ and relations as in \cite[Section 6.1.2]{klsc}.
The representation theory of $\Uqg$ is essentially the same as for $U(\lieg)$. This means that we have a simple $\Uqg$-module denoted by $V(\lambda)$ for each dominant integral weight $\lambda$.

Dually we have the \emph{quantized coordinate ring} $\CqG$.
This is a subspace of the linear dual of $\Uqg$, namely the space of matrix coefficients of all finite-dimensional (type 1) $\Uqg$-modules, see for instance \cite[Section 2.2.6]{heko}.
By construction $\CqG$ is a Hopf algebra and a $\Uqg$-bimodule.
These actions can be used to define quantum homogeneous spaces as appropriate subalgebras of $\CqG$, for example the quantum projective spaces.

\subsection{Braiding}

In what follows, an important role is going to be played by the monoidal category of finite-dimensional $\Uqg$-modules.
This is a braided monoidal category, where we recall that a \emph{braiding} on a monoidal category is the choice of a natural isomorphism $X \otimes Y \cong Y \otimes X$ for each pair of objects $X$ and $Y$, satisfying the hexagon relations \cite[Definition 8.1.1]{egno}.
It is a generalization of the flip map in the category of vector spaces.

For the category of finite-dimensional $\Uqg$-modules we write the braiding as
\[
\braid_{V, W}: V \otimes W \to W \otimes V.
\]
The braiding is not unique and we adopt the same choice as \cite{heko}, described as follows.
Given two simple modules $V(\lambda)$ and $V(\mu)$, choose a highest weight vector $v_\lambda$ for the first and a lowest weight vector $w_{w_0 \mu}$ for the second (here $w_0$ denotes the longest word of the Weyl group, as usual).
Then the braiding is determined by
\[
\braid_{V(\lambda), V(\mu)} (v_\lambda \otimes w_{w_0 \mu}) = q^{(\lambda, w_0 \mu)} w_{w_0 \mu} \otimes v_\lambda.
\]
Indeed, $v_\lambda \otimes w_{w_0 \mu}$ is a cyclic vector for $V(\lambda) \otimes V(\mu)$, hence $\braid_{V(\lambda), V(\mu)}$ is completely determined by the action on this vector and the fact that it is a $\Uqg$-module map.
We also note that this implies the following: given weight vectors $v_i \in V(\lambda)$ and $w_j \in V(\mu)$, we have
\begin{equation}
\label{eq:braiding-terms}
\braid_{V(\lambda), V(\mu)} (v_i \otimes w_j) = q^{(\wt(v_i), \wt(w_j))} w_j \otimes v_i + \cdots,
\end{equation}
where any additional term in this expression has weight lower than $\wt(w_j)$ in the first leg and weight higher than $\wt(v_i)$ in the second leg.

\subsection{Duality}

The notion of duality in a monoidal category is captured by the existence of \emph{evaluation} and \emph{coevaluation} morphisms, which in our setting are described as follows.
Let $V$ be a finite-dimensional $\Uqg$-module and $V^*$ its linear dual.
Then, omitting the dependence on $V$ (which will be fixed in our computations), we have $\Uqg$-module maps
\[
\begin{gathered}
\EV: V^* \otimes V \to \bbC, \quad
\CV: \bbC \to V \otimes V^*, \\
\EVp: V \otimes V^* \to \bbC, \quad
\CVp: \bbC \to V^* \otimes V,
\end{gathered}
\]
satisfying the following duality relations (using standard leg-notation)
\begin{equation}
\label{eq:duality}
\begin{gathered}
\EV_{2 3} \CV_1 = \id, \quad
\EV_{1 2} \CV_2 = \id, \\
\EVp_{2 3} \CVp_1 = \id, \quad
\EVp_{1 2} \CVp_2 = \id.
\end{gathered}
\end{equation}
The maps $\EV$ and $\CV$ correspond to the existence of a \emph{left dual}, while $\EVp$ and $\CVp$ correspond to the existence of a \emph{right dual}.
In the case of the quantized enveloping algebra $\Uqg$, the property $S^2(X) = K_{2 \rho} X K_{2 \rho}^{-1}$ of the antipode guarantees that the two duals can be identified.

Let us now discuss the explicit formulae for these maps.
Take a weight basis $\{v_i\}_i$ of $V$ and write $\lambda_i := \wt(v_i)$ for the weight of $v_i$. For the corresponding dual basis $\{f^i\}_i$ of $V^*$ we have $\wt(f^i) = -\lambda_i$.
Then the evaluation and coevaluation maps are given by
\[
\begin{gathered}
\EV(f^i \otimes v_j) = \delta^i_j, \quad
\CV = \sum_i v_i \otimes f^i, \\
\EVp(v_i \otimes f^j) = q^{(\lambda_i, 2 \rho)} \delta^j_i, \quad
\CVp = \sum_i q^{-(\lambda_i, 2 \rho)} f^i \otimes v_i.
\end{gathered}
\]
The factor $q^{(\lambda_i, 2 \rho)}$ is related to the element $K_{2 \rho}$ in the square of the antipode.

We also have various compatibility relations with the braiding $\braid_{V, W}$, since the latter is a \emph{natural isomorphism} in both entries.
For the evaluation morphisms we have
\begin{equation}
\label{eq:evaluations}
\begin{gathered}
\EV_{1 2} = \EV_{2 3} (\braid_{V^*, W})_{1 2} (\braid_{V, W})_{2 3}, \quad
\EV_{2 3} = \EV_{1 2} (\braid_{W, V})_{2 3} (\braid_{W, V^*})_{1 2}, \\
\EVp_{1 2} = \EVp_{2 3} (\braid_{V, W})_{1 2} (\braid_{V^*, W})_{2 3}, \quad
\EVp_{2 3} = \EVp_{1 2} (\braid_{W, V^*})_{2 3} (\braid_{W, V})_{1 2}.
\end{gathered}
\end{equation}
Similarly, for the coevaluations morphisms we have
\begin{equation}
\label{eq:coevaluations}
\begin{gathered}
\CV_1 = (\braid_{W, V^*})_{2 3} (\braid_{W, V})_{1 2} \CV_2, \quad
\CV_2 = (\braid_{V, W})_{1 2} (\braid_{V^*, W})_{2 3} \CV_1, \\
\CVp_1 = (\braid_{W, V})_{2 3} (\braid_{W, V^*})_{1 2} \CVp_2, \quad
\CVp_2 = (\braid_{V^*, W})_{1 2} (\braid_{V, W})_{2 3} \CVp_1.
\end{gathered}
\end{equation}

We are also going to need the following identities.

\begin{lemma}
\label{lem:categorical-identities}
Let $V = V(\lambda)$ be a simple $\Uqg$-module. Then we have
\[
\EVp = q^{(\lambda, \lambda + 2 \rho)} \EV \circ \braid_{V, V^*}, \quad
\CVp = q^{(\lambda, \lambda + 2 \rho)} \braid_{V, V^*} \circ \CV.
\]
\end{lemma}

\begin{proof}
We have that $\EVp$ and $\EV \circ \braid_{V, V^*}$ are both morphisms from $V \otimes V^*$ to $\bbC$.
Since $V$ is a simple module we must have $\EVp = c \EV \circ \braid_{V, V^*}$ for some $c \in \bbC$.
Consider $v_\lambda \otimes f^{-\lambda}$, where $v_\lambda$ is a highest weight vector of $V$ and $f^{-\lambda}$ is its dual, which is a lowest weight vector of $V^*$.
Then we have $\EVp(v_\lambda \otimes f^{-\lambda}) = q^{(\lambda, 2 \rho)}$. On the other hand we have
\[
\EV \circ \braid_{V, V^*}(v_\lambda \otimes f^{-\lambda})
= q^{-(\lambda, \lambda)} \EV (f^{-\lambda} \otimes v_\lambda) = q^{-(\lambda, \lambda)}.
\]
Comparing the two expressions we find $c = q^{(\lambda, \lambda + 2 \rho)}$.

Similarly we have that $\CVp$ and $\braid_{V, V^*} \circ \CV$ are both maps from $\bbC$ to $V^* \otimes V$.
Since $V$ is a simple module we must have $\CVp = c^\prime \braid_{V, V^*} \circ \CV$. Now consider $v_{w_0 \lambda} \otimes f^{-w_0 \lambda}$, where $v_{w_0 \lambda}$ is a lowest weight vector of $V$ and $f^{-w_0 \lambda}$ is its dual, which is a highest weight vector of $V^*$.
Taking into account the property \eqref{eq:braiding-terms} for the braiding, we have
\[
\braid_{V, V^*}(v_{w_0 \lambda} \otimes f^{-w_0 \lambda}) = q^{-(w_0 \lambda, w_0 \lambda)} f^{-w_0 \lambda} \otimes v_{w_0 \lambda} + \cdots.
\]
Observe that $f^{-w_0 \lambda} \otimes v_{w_0 \lambda}$ can not appear in the additional terms.
By the same argument the term $f^{-w_0 \lambda} \otimes v_{w_0 \lambda}$ can not appear in $\braid_{V, V^*}(v_i \otimes f^j)$ unless $v_i \otimes f^j = v_{w_0 \lambda} \otimes f^{-w_0 \lambda}$.

Since $\CV = \sum_i v_i \otimes f^i$, the term $f^{-w_0 \lambda} \otimes v_{w_0 \lambda}$ appears in $\braid_{V, V^*} \circ \CV$ with coefficient $q^{-(w_0 \lambda, w_0 \lambda)} = q^{-(\lambda, \lambda)}$, by the argument above.
On the other hand in $\CVp = \sum_i q^{-(\lambda_i, 2 \rho)} f^i \otimes v_i$ it appears with coefficient $q^{-(w_0 \lambda, 2 \rho)} = q^{-(\lambda, 2 w_0 \rho)}$.
We have $-w_0 \rho = \rho$, since $-w_0$ acts by permuting the fundamental weights and $\rho$ is the sum of all the fundamental weights.
Hence we obtain $q^{-(w_0 \lambda, 2 \rho)} = q^{(\lambda, 2 \rho)}$.
Comparing the two expressions we get $c^\prime = q^{(\lambda, \lambda + 2 \rho)}$.
\end{proof}

\subsection{Quadratic condition}

In the following we shall focus on the case of $U_q(\mathfrak{sl}_{r + 1})$ and fix $V = V(\omega_s)$ with $s = 1$ or $s = r$, corresponding to the fundamental module of $U_q(\mathfrak{sl}_{r + 1})$ or its dual.
The index $s$ is always going to be used in this sense.

With $V$ as above, we have that the braiding $\braid_{V, V}$ satisfies a quadratic relation, known as the \emph{Hecke relation} in this context. More precisely this is
\begin{equation}
\label{eq:hecke-relation}
(\braid_{V, V} - q^{(\omega_s, \omega_s)}) (\braid_{V, V} + q^{(\omega_s, \omega_s) - (\alpha_s, \alpha_s)}) = 0.
\end{equation}
This reflects the fact that the tensor product $V(\omega_s) \otimes V(\omega_s)$ has a decomposition into simple modules with only two summands, namely
\[
V(\omega_s) \otimes V(\omega_s) \cong V(2 \omega_s) \oplus V(2 \omega_s - \alpha_s).
\]
Note that both $\braid_{V, V}$ and $\braid_{V^*, V^*}$ satisfy such quadratic conditions.

\section{Results on quantum projective spaces}
\label{sec:projective-spaces}

In this section we recall various results related to quantum projective spaces.
We begin with the presentation of their algebras and the differential calculi of Heckenberger-Kolb \cite{heko}.
Next, we consider the quantum Riemannian geometry framework of \cite{quantum-book}, in particular the notions of quantum metric, inverse metric and quantum Levi-Civita connection.
These were all derived in \cite{paper-projective} for general quantum projective spaces, generalizing previous results for the quantum $2$-sphere.
The cited paper forms the technical backbone of the current one, to which we refer for some occasional unexplained definitions and notations used here.

\subsection{Quantum projective space}

We denote by $\subalg$ the algebra of functions on the $r$-dimensional quantum (complex) projective space, which is a subalgebra of $\bbC_q[SL_{r + 1}]$ generated by certain elements $\{p^{i j}\}_{i, j = 1}^{r +1}$ (see for instance \cite{paper-projective} for details).
The essential thing to know is that they correspond to the product of certain matrix coefficients of two $U_q(\mathfrak{sl}_{r + 1})$-modules, namely $V := V(\omega_s)$ (with $s = 1$ or $s = r$) and its dual $V^*$.

To write the relations between these elements, we introduce the maps
\begin{equation}
\label{eq:S-maps}
\begin{split}
\opS_{1 2 3} & := (\braid_{V, V^*})_{2 3} (\braid_{V, V})_{1 2} (\braid^{-1}_{V, V^*})_{2 3}, \\
\opSt_{2 3 4} & := (\braid_{V, V^*})_{2 3} (\braid^{-1}_{V^*, V^*})_{3 4} (\braid^{-1}_{V, V^*})_{2 3}.
\end{split}
\end{equation}
The main properties satisfied by these maps are recalled in \cref{sec:properties-S}.
Then, using index-free notation as in \cite{heko} and \cite{paper-projective}, we have the relations
\begin{equation}
\label{eq:algebra-relations}
\opS_{1 2 3} p p = q^{(\omega_s, \omega_s)} p p, \quad
\opSt_{2 3 4} p p = q^{-(\omega_s, \omega_s)} p p, \quad
\EVp_{1 2} p = q^{(\omega_s, 2 \rho)}.
\end{equation}
In this way we obtain presentation of the algebra $\subalg$ by generators and relations, as proven in \cite{heko}.
In what follows we shall mostly forget about the underlying quantum groups, but crucially use the categorical relations for the $U_q(\mathfrak{sl}_{r + 1})$-modules $V$ and $V^*$.

We should mention that, using \eqref{eq:algebra-relations}, it easily follows that the generators $p^{i j}$ are the entries of a projection.
This means that $\sum_k p^{i k} p^{k j} = p^{i j}$, or in index-free notation
\begin{equation}
\label{eq:projection-relation}
\EV_{2 3} p p = p.
\end{equation}
It can be derived by applying $\EVp_{1 2}$ to the first relation of \eqref{eq:algebra-relations} and using \eqref{eq:S-evaluation}.

\subsection{Heckenberger-Kolb calculus}
Now we describe the first-order differential calculus (abbreviated as FODC) on $\subalg$ introduced by Heckenberger and Kolb in \cite{heko}.
We denote it by $\calc$ and follow the presentation given in \cite{paper-projective}.
It is a direct sum $\calc := \calcP \oplus \calcM$, with $\calcP$ and $\calcM$ generated respectively by $\del p$ and $\delbar p$ as left $\subalg$-modules, and with relations
\begin{equation}
\label{eq:calculus-relations}
\begin{gathered}
\opSt_{2 3 4} p \del p = q^{-(\omega_s, \omega_s)} p \del p, \quad
\EVp_{1 2} \del p = 0, \\
\opS_{1 2 3} p \delbar p = q^{(\omega_s, \omega_s)} p \delbar p, \quad
\EVp_{1 2} \delbar p = 0.
\end{gathered}
\end{equation}

To describe the right $\subalg$-module structure of $\calc$ it is convenient to write
\begin{equation}
\label{eq:T-map}
\opT_{1 2 3 4} := \opS_{1 2 3} \opSt_{2 3 4} = \opSt_{2 3 4} \opS_{1 2 3}.
\end{equation}
Then the right $\subalg$-module relations are given by
\begin{equation}
\label{eq:right-module-relations}
\begin{split}
\del p p & = q^{(\alpha_s, \alpha_s)} \opT_{1 2 3 4} p \del p = q^{(\alpha_s, \alpha_s) - (\omega_s, \omega_s)} \opS_{1 2 3} p \del p, \\
\delbar p p & = q^{-(\alpha_s, \alpha_s)} \opT_{1 2 3 4} p \delbar p = q^{(\omega_s, \omega_s) - (\alpha_s, \alpha_s)} \opSt_{2 3 4} p \delbar p.
\end{split}
\end{equation}
Here the alternative expressions are obtained using \eqref{eq:calculus-relations}.
We also note that, using the relations above, we can rewrite those of \eqref{eq:calculus-relations} with $p$ appearing on the right.

Finally one can obtain the following simple relations for $\del p$ and $\delbar p$:
\begin{equation}
\label{eq:evaluations-calculi}
\begin{gathered}
\EV_{2 3} p \del p = 0, \quad
\EV_{2 3} \del p p = \del p, \\
\EV_{2 3} p \delbar p = \delbar p, \quad
\EV_{2 3} \delbar p p = 0.
\end{gathered}
\end{equation}
For more details about these, see for instance \cite[Lemma 5.2]{kahler} (with different notation).

\subsection{Fubini-Study metric}

The natural metric on a projective space is given by the \emph{Fubini-Study metric}.
A quantum analogue of this metric is given by $\met := \metPM + \metMP$, where
\begin{equation}
\label{eq:quantum-metric}
\metPM := \EVp_{1 2} \EV_{2 3} \del p \otimes \delbar p, \quad
\metMP := \EVp_{1 2} \EV_{2 3} \delbar p \otimes \del p.
\end{equation}
It can be seen that $\met$ reduces to the Fubini-Study metric in the classical case.

We have that $\met$ is a \emph{quantum metric} in the sense of Beggs-Majid \cite{quantum-book}.
This notion requires the existence of an appropriate \emph{inverse metric}, which amounts to a $\subalg$-bimodule map $(\cdot, \cdot): \calc \otimes_\subalg \calc \to \subalg$ such that the following conditions hold
\[
(\imet \otimes \id) (\omega \otimes \met) = \omega = (\id \otimes \imet) (\met \otimes \omega), \quad \forall \omega \in \calc.
\]

As shown in \cite[Theorem 6.11]{paper-projective}, the inverse metric corresponding to the Fubini-Study metric $\met$ is given by $(\del p, \del p) = (\delbar p, \delbar p) := 0$ and by
\begin{equation}
\label{eq:inverse-metric}
\begin{split}
(\del p, \delbar p) & := q^{(\alpha_s, \alpha_s)} q^{-(\omega_s, \omega_s + 2 \rho)} \opS_{1 2 3} \CV_3 p - q^{(\alpha_s, \alpha_s)} q^{-(\omega_s, 2 \rho)} p p, \\
(\delbar p, \del p) & := \CVp_2 p - q^{-(\omega_s, 2 \rho)} p p.
\end{split}
\end{equation}
Further properties of the quantum metric $\met$ are recalled later on.

\subsection{Levi-Civita connection}

According to \cite[Proposition 7.1 and Proposition 7.2]{paper-projective} we have a (left) connection $\conn: \calc \to \calc \otimes_\subalg \calc$ defined by
\begin{equation}
\label{eq:connection}
\begin{split}
\conn(\del p) & := q^{(\alpha_s, \alpha_s)} \EV_{2 3} \opS_{1 2 3} \opSt_{2 3 4} \delbar p \otimes \del p - q^{(\alpha_s, \alpha_s)} q^{-(\omega_s, 2 \rho)} p \metMP, \\
\conn(\delbar p) & := \EV_{2 3} \del p \otimes \delbar p - q^{-(\omega_s, 2 \rho)} p \metPM.
\end{split}
\end{equation}
In the classical case it reduces to the \emph{Levi-Civita connection} with respect to the Fubini-Study metric (acting on the cotangent bundle).
In the quantum case it satisfies the conditions of a \emph{quantum Levi-Civita connection} in the sense of \cite{quantum-book}, as shown in \cite[Theorem 8.4]{paper-projective}.
This essentially means that $\conn$ is torsion-free and compatible with the quantum metric $\met$.
In particular this implies that that $\conn$ is a bimodule connection (see below for this definition).

We should also mention that, due to the results of \cite{connections-irreducible}, this connection coincides with the unique \emph{covariant} connection on $\calc$ (and hence also with the one in \cite{projective-bundle}).

\subsection{Generalized braiding}

We recall that $\conn$ being a \emph{bimodule connection} means that there exists a $\subalg$-bimodule map $\sigma: \calc \otimes_\subalg \calc \to \calc \otimes_\subalg \calc$, called the \emph{generalized braiding}, such that
\[
\conn(\omega b) = \sigma(\omega \otimes \diff b) + \conn(\omega) b, \quad
\omega \in \calc, \ b \in \subalg.
\]
This map is unique, when it exists.
The main property of bimodule connections is that they naturally lift to connections on tensor products (in particular to $\calc \otimes_\subalg \calc$).

The quantum Levi-Civita connection $\conn$ is a bimodule connection, as shown in \cite[Proposition 8.3]{paper-projective}.
The corresponding generalized braiding $\sigma$ is given by
\begin{equation}
\label{eq:generalized-braiding}
\begin{split}
\sigma(\del p \otimes \del p) & := q^{(\alpha_s, \alpha_s)} \opS_{1 2 3} \opSt_{2 3 4} \del p \otimes \del p, \\
\sigma(\delbar p \otimes \delbar p) & := q^{-(\alpha_s, \alpha_s)} \opS_{1 2 3} \opSt_{2 3 4} \delbar p \otimes \delbar p, \\
\sigma(\del p \otimes \delbar p) & := q^{2 (\alpha_s, \alpha_s) - 2 (\omega_s, \omega_s)} \opS_{1 2 3} \opSt_{2 3 4}^{-1} \delbar p \otimes \del p + (1 - q^{(\alpha_s, \alpha_s)}) q^{(\alpha_s, \alpha_s)} q^{-(\omega_s, 2 \rho)} p \metMP p, \\
\sigma(\delbar p \otimes \del p) & := q^{2 (\omega_s, \omega_s) - 2 (\alpha_s, \alpha_s)} \opS_{1 2 3}^{-1} \opSt_{2 3 4} \del p \otimes \delbar p + (1 - q^{-(\alpha_s, \alpha_s)}) q^{-(\omega_s, 2 \rho)} p \metPM p.
\end{split}
\end{equation}
In the classical case it simply reduces to the flip map.

In the following we write $\sigma_{a b}: \calc_a \otimes_\subalg \calc_b \to \calc_b \otimes_\subalg \calc_a$ (with $a, b \in \{ +, - \}$) for the restriction of the map $\sigma$ to $\calc_a \otimes_\subalg \calc_b$. Note that these are all $\subalg$-bimodule maps.

\subsection{Properties of the metric}

Finally we recall some additional properties satisfied by the quantum metric $\met$.
It is \emph{symmetric}, in the sense that
\begin{equation}
\label{eq:metric-symmetric}
\wedge(\met) = 0.
\end{equation}
Here $\wedge: \calc \otimes_\subalg \calc \to \calc^2$ denotes the product of one-forms, as usual.
Observe that, given the decomposition of $\met$, this implies the identity $\wedge(\metPM) = - \wedge(\metMP)$.

Classically the Fubini-Study metric on a complex projective space is a \emph{Kähler metric}.
This statement is also true in the quantum case, once suitably interpreted.
For our purposes this means the identities (see the proof of \cite[Proposition 6.5]{paper-projective})
\begin{equation}
\label{eq:kahler-property}
(\diff \otimes \id)(\met_{a b}) = (\id \otimes \diff)(\met_{a b}) = 0, \quad a, b \in \{ +, - \}.
\end{equation}

Finally we have some identities related the \emph{compatibility} of the quantum Levi-Civita connection $\nabla$ with the quantum metric $\met$.
These are (see the proof of \cite[Theorem 8.4]{paper-projective})
\begin{equation}
\label{eq:metric-compatibility}
(\conn \otimes \id)(\met_{a b}) = (\id \otimes \conn)(\met_{a b}) = 0, \quad a, b \in \{ +, - \}.
\end{equation}

\section{Properties of the generalized braiding}
\label{sec:generalized-braiding}

In this section we prove various properties satisfied by the generalized braiding $\sigma$.
First, we show that its components satisfy certain quadratic relations.
Then we use this result to give a new presentation of the Heckenberger-Kolb calculus $\calc^\bullet$: we show that it can be obtained from the tensor algebra of the FODC $\calc$ upon taking the quotient by the ideal generated by $\mathrm{im}(\id + \sigma)$.
This parallels the classical construction of $\calc^\bullet$ from the space of one-forms $\calc$.

\subsection{Quadratic relations}

Our goal here is to show that the components of the generalized braiding $\sigma_{a b}$ (with $a, b \in \{ +, - \}$) satisfy certain quadratic relations.
We begin by computing the action of $\sigma$ on the components of the quantum metric $\met$.

\begin{lemma}
We have
\begin{equation}
\label{eq:metric-sigma}
\sigma(\metPM) = \metMP, \quad
\sigma(\metMP) = \metPM.
\end{equation}
In particular $\sigma(\met) = \met$.
\end{lemma}

\begin{proof}
\textit{The term $\sigma(\metPM)$.}
Applying the generalized braiding \eqref{eq:generalized-braiding} to $\metPM$ we get
\[
\begin{split}
\sigma(\metPM) & = q^{2 (\alpha_s, \alpha_s) - 2 (\omega_s, \omega_s)} \EVp_{1 2} \EV_{2 3} \opS_{1 2 3} \opSt_{2 3 4}^{-1} \delbar p \otimes \del p \\
& + (1 - q^{(\alpha_s, \alpha_s)}) q^{(\alpha_s, \alpha_s)} q^{-(\omega_s, 2 \rho)} \EVp_{1 2} \EV_{2 3} p \metMP p.
\end{split}
\]
Let us abbreviate this as
\[
\sigma(\metPM) = q^{2 (\alpha_s, \alpha_s) - 2 (\omega_s, \omega_s)} A_1 + (1 - q^{(\alpha_s, \alpha_s)}) q^{(\alpha_s, \alpha_s)} q^{-(\omega_s, 2 \rho)} A_2.
\]

Consider the term $A_1$. Using the quadratic relation \eqref{eq:S-quadratic-relations} for $\opS_{1 2 3}$ we get
\[
\begin{split}
A_1 & = q^{2 (\omega_s, \omega_s) - (\alpha_s, \alpha_s)} \EVp_{1 2} \EV_{2 3} \opS_{1 2 3}^{-1} \opSt_{2 3 4}^{-1} \delbar p \otimes \del p + (1 - q^{-(\alpha_s, \alpha_s)}) q^{(\omega_s, \omega_s)} \EVp_{1 2} \EV_{2 3} \opSt_{2 3 4}^{-1} \delbar p \otimes \del p.
\end{split}
\]
Using \eqref{eq:S-quadratic-relations} once more for $\opSt_{2 3 4}^{-1}$ we get
\[
\begin{split}
A_1 & = q^{2 (\omega_s, \omega_s) - (\alpha_s, \alpha_s)} \EVp_{1 2} \EV_{2 3} \opS_{1 2 3}^{-1} \opSt_{2 3 4}^{-1} \delbar p \otimes \del p \\
& + (1 - q^{-(\alpha_s, \alpha_s)}) q^{3 (\omega_s, \omega_s) - (\alpha_s, \alpha_s)} \EVp_{1 2} \EV_{2 3} \opSt_{2 3 4} \delbar p \otimes \del p \\
& + (1 - q^{-(\alpha_s, \alpha_s)})^2 q^{2 (\omega_s, \omega_s)} \EVp_{1 2} \EV_{2 3} \delbar p \otimes \del p.
\end{split}
\]
The second term vanishes due to the identities $\EV_{2 3} \opSt_{2 3 4} = q^{-(\omega_s, \omega_s + 2 \rho)} \EVp_{3 4}$ and $\EVp_{3 4} \delbar p \otimes \del p = 0$, which correspond to \eqref{eq:S-evaluation} and \eqref{eq:tensor-algebra-evaluations} respectively. Then
\[
\begin{split}
A_1 & = q^{2 (\omega_s, \omega_s) - (\alpha_s, \alpha_s)} \EVp_{1 2} \EV_{2 3} \opS_{1 2 3}^{-1} \opSt_{2 3 4}^{-1} \delbar p \otimes \del p + q^{2 (\omega_s, \omega_s)} (1 - q^{-(\alpha_s, \alpha_s)})^2 \EVp_{1 2} \EV_{2 3} \delbar p \otimes \del p.
\end{split}
\]
Now using the relation \eqref{eq:S-evaluation} twice we have
\[
\EVp_{1 2} \EV_{2 3} \opS_{1 2 3}^{-1} \opSt_{2 3 4}^{-1} = q^{-(\omega_s, \omega_s + 2 \rho)} \EVp_{1 2} \EVp_{1 2} \opSt_{2 3 4}^{-1}
= q^{-(\omega_s, \omega_s + 2 \rho)} \EVp_{1 2} \EVp_{3 4} \opSt_{2 3 4}^{-1} = \EVp_{1 2} \EV_{2 3}.
\]
Therefore we obtain
\[
\begin{split}
A_1 & = q^{2 (\omega_s, \omega_s) - (\alpha_s, \alpha_s)} \EVp_{1 2} \EV_{2 3} \delbar p \otimes \del p + q^{2 (\omega_s, \omega_s)} (1 - q^{-(\alpha_s, \alpha_s)})^2 \EVp_{1 2} \EV_{2 3} \delbar p \otimes \del p \\
& = q^{2 (\omega_s, \omega_s) - 2 (\alpha_s, \alpha_s)} (q^{2 (\alpha_s, \alpha_s)} - q^{(\alpha_s, \alpha_s)} + 1) \metMP.
\end{split}
\]
Now consider the term $A_2$. Taking into account that the component $\metMP$ of the quantum metric is central, we use the algebra relations to compute
\[
A_2 = \EVp_{1 2} \EV_{2 3} p \metMP p = \EVp_{1 2} \EV_{2 3} p p \metMP
= \EVp_{1 2} p \metMP = q^{(\omega_s, 2 \rho)} \metMP.
\]
Finally we obtain
\[
\sigmaPM(\metPM) = (q^{2 (\alpha_s, \alpha_s)} - q^{(\alpha_s, \alpha_s)} + 1) \metMP - (q^{2 (\alpha_s, \alpha_s)} - q^{(\alpha_s, \alpha_s)}) \metMP = \metMP.
\]

\textit{The term $\sigma(\metMP)$.}
Similarly applying \eqref{eq:generalized-braiding} to $\metMP$ we get
\[
\begin{split}
\sigma(\metMP) & = q^{2 (\omega_s, \omega_s) - 2 (\alpha_s, \alpha_s)} \EVp_{1 2} \EV_{2 3} \opS_{1 2 3}^{-1} \opSt_{2 3 4} \del p \otimes \delbar p \\
& + (1 - q^{-(\alpha_s, \alpha_s)}) q^{-(\omega_s, 2 \rho)} \EVp_{1 2} \EV_{2 3} p \metPM p.
\end{split}
\]
Let us abbreviate this as
\[
\sigma(\metMP) = q^{2 (\omega_s, \omega_s) - 2 (\alpha_s, \alpha_s)} B_1 + (1 - q^{-(\alpha_s, \alpha_s)}) q^{-(\omega_s, 2 \rho)} B_2.
\]

Consider the term $B_1$. Using \eqref{eq:S-evaluation} and \eqref{eq:S-quadratic-relations} we compute
\[
\begin{split}
B_1 & = q^{-(\omega_s, \omega_s + 2 \rho)} \EVp_{1 2} \EVp_{1 2} \opSt_{2 3 4} \del p \otimes \delbar p
= q^{-(\omega_s, \omega_s + 2 \rho)} \EVp_{1 2} \EVp_{3 4} \opSt_{2 3 4} \del p \otimes \delbar p \\
& = q^{(\alpha_s, \alpha_s) - 2 (\omega_s, \omega_s)} q^{-(\omega_s, \omega_s + 2 \rho)} \EVp_{1 2} \EVp_{3 4} \opSt_{2 3 4}^{-1} \del p \otimes \delbar p \\
& = q^{(\alpha_s, \alpha_s) - 2 (\omega_s, \omega_s)} \EVp_{1 2} \EV_{2 3} \del p \otimes \delbar p = q^{(\alpha_s, \alpha_s) - 2 (\omega_s, \omega_s)} \metPM.
\end{split}
\]
For the term $B_2$ we compute
\[
B_2 = \EVp_{1 2} \EV_{2 3} p p \metPM
= \EVp_{1 2} p \metPM
= q^{(\omega_s, 2 \rho)} \metPM.
\]
Therefore we obtain
\[
\begin{split}
\sigmaMP(\metMP) & = q^{2 (\omega_s, \omega_s) - 2 (\alpha_s, \alpha_s)} q^{(\alpha_s, \alpha_s) - 2 (\omega_s, \omega_s)} \metPM - (q^{-(\alpha_s, \alpha_s)} - 1) q^{-(\omega_s, 2 \rho)} q^{(\omega_s, 2 \rho)} \metPM \\
& = q^{-(\alpha_s, \alpha_s)} \metPM - (q^{-(\alpha_s, \alpha_s)} - 1) \metPM = \metPM. \qedhere
\end{split}
\]
\end{proof}

We are now ready to derive the quadratic relations for the maps $\sigma_{a b}$.

\begin{proposition}
\label{prop:quadratic-relations}
We have the identities
\[
\begin{gathered}
(\sigmaPP - q^{(\alpha_s, \alpha_s)} \id) \circ (\sigmaPP + \id) = 0, \\
(\sigmaMM - q^{-(\alpha_s, \alpha_s)} \id) \circ (\sigmaMM + \id) = 0, \\
\sigmaMP \circ \sigmaPM = \id, \quad \sigmaPM \circ \sigmaMP = \id.
\end{gathered}
\]
\end{proposition}

\begin{proof}
\textit{The maps $\sigmaPP$ and $\sigmaMM$.}
We begin with the map $\sigmaPP$.
Using the identity \eqref{eq:tensor-algebra-identities} we can rewrite the expression for $\sigmaPP$ in the following way
\[
\sigma(\del p \otimes \del p) = q^{(\alpha_s, \alpha_s)} \opT_{1 2 3 4} \del p \otimes \del p
= q^{(\alpha_s, \alpha_s) - (\omega_s, \omega_s)} \opS_{1 2 3} \del p \otimes \del p.
\]
Next, we rewrite the quadratic relation from \eqref{eq:S-quadratic-relations} in the form
\[
\opS_{1 2 3}^2 = (1 - q^{-(\alpha_s, \alpha_s)}) q^{(\omega_s, \omega_s)} \opS_{1 2 3} + q^{2 (\omega_s, \omega_s) - (\alpha_s, \alpha_s)}.
\]
Using this identity we compute
\[
\begin{split}
\sigmaPP^2(\del p \otimes \del p)
& = q^{2 (\alpha_s, \alpha_s) - 2 (\omega_s, \omega_s)} \opS_{1 2 3}^2 \del p \otimes \del p \\
& = (q^{(\alpha_s, \alpha_s)} - 1) q^{(\alpha_s, \alpha_s) - (\omega_s, \omega_s)} \opS_{1 2 3} \del p \otimes \del p + q^{(\alpha_s, \alpha_s)} \del p \otimes \del p \\
& = (q^{(\alpha_s, \alpha_s)} - 1) \sigmaPP(\del p \otimes \del p) + q^{(\alpha_s, \alpha_s)} \del p \otimes \del p.
\end{split}
\]
This expression can be rewritten in the form $(\sigmaPP - q^{(\alpha_s, \alpha_s)}) (\sigmaPP + 1) = 0$.

The case of $\sigmaMM$ is very similar. We use \eqref{eq:tensor-algebra-identities} again to rewrite
\[
\sigmaMM(\delbar p \otimes \delbar p) = q^{-(\alpha_s, \alpha_s)} \opT_{1 2 3 4} \delbar p \otimes \delbar p
= q^{(\omega_s, \omega_s) - (\alpha_s, \alpha_s)} \opSt_{2 3 4} \delbar p \otimes \delbar p.
\]
Moreover from \eqref{eq:S-quadratic-relations} we have
\[
\opSt_{2 3 4}^2 = (1 - q^{(\alpha_s, \alpha_s)}) q^{-(\omega_s, \omega_s)} \opSt_{2 3 4} + q^{(\alpha_s, \alpha_s) - 2 (\omega_s, \omega_s)}.
\]
A quick computation then leads to $(\sigmaMM - q^{-(\alpha_s, \alpha_s)}) (\sigmaMM + 1) = 0$.

\textit{The maps $\sigmaPM$ and $\sigmaMP$.}
Now let us consider the maps
\[
\begin{split}
\sigmaPM(\del p \otimes \delbar p) & = q^{2 (\alpha_s, \alpha_s) - 2 (\omega_s, \omega_s)} \opS_{1 2 3} \opSt_{2 3 4}^{-1} \delbar p \otimes \del p + (1 - q^{(\alpha_s, \alpha_s)}) q^{(\alpha_s, \alpha_s)} q^{-(\omega_s, 2 \rho)} p \metMP p, \\
\sigmaMP(\delbar p \otimes \del p) & = q^{2 (\omega_s, \omega_s) - 2 (\alpha_s, \alpha_s)} \opS_{1 2 3}^{-1} \opSt_{2 3 4} \del p \otimes \delbar p + (1 - q^{-(\alpha_s, \alpha_s)}) q^{-(\omega_s, 2 \rho)} p \metPM p.
\end{split}
\]
Taking into account \eqref{eq:metric-sigma} and that $\sigma$ is a $\subalg$-bimodule map we compute
\[
\begin{split}
\sigmaMP \sigmaPM(\del p \otimes \delbar p)
& = q^{2 (\alpha_s, \alpha_s) - 2 (\omega_s, \omega_s)} \opS_{1 2 3} \opSt_{2 3 4}^{-1} \sigmaMP(\delbar p \otimes \del p) \\
& + (1 - q^{(\alpha_s, \alpha_s)}) q^{(\alpha_s, \alpha_s)} q^{-(\omega_s, 2 \rho)} p \sigmaMP(\metMP) p \\
& = \del p \otimes \delbar p + (1 - q^{-(\alpha_s, \alpha_s)}) q^{2 (\alpha_s, \alpha_s) - 2 (\omega_s, \omega_s)} q^{-(\omega_s, 2 \rho)} \opS_{1 2 3} \opSt_{2 3 4}^{-1} p \metPM p \\
& + (1 - q^{(\alpha_s, \alpha_s)}) q^{(\alpha_s, \alpha_s)} q^{-(\omega_s, 2 \rho)} p \metPM p.
\end{split}
\]
We have $p \metPM p = p p \metPM$, since $\met$ is central. Moreover using \eqref{eq:algebra-relations} we have
\[
\opS_{1 2 3} \opSt_{2 3 4}^{-1} p p = q^{(\omega_s, \omega_s)} \opS_{1 2 3} p p = q^{2 (\omega_s, \omega_s)} p p.
\]
Then $\opS_{1 2 3} \opSt_{2 3 4}^{-1} p \metPM p = q^{2 (\omega_s, \omega_s)} p \metPM p$ and hence
\[
\begin{split}
\sigmaMP \sigmaPM(\del p \otimes \delbar p)
& = \del p \otimes \delbar p - (1 - q^{(\alpha_s, \alpha_s)}) q^{(\alpha_s, \alpha_s)} q^{-(\omega_s, 2 \rho)} p \metPM p \\
& + (1 - q^{(\alpha_s, \alpha_s)}) q^{(\alpha_s, \alpha_s)} q^{-(\omega_s, 2 \rho)} p \metPM p \\
& = \del p \otimes \delbar p.
\end{split}
\]

Similarly we check that $\sigmaPM \sigmaMP = \id$. We compute
\[
\begin{split}
\sigmaPM \sigmaMP(\delbar p \otimes \del p)
& = q^{2 (\omega_s, \omega_s) - 2 (\alpha_s, \alpha_s)} \opS_{1 2 3}^{-1} \opSt_{2 3 4} \sigmaPM(\del p \otimes \delbar p) \\
& + (1 - q^{-(\alpha_s, \alpha_s)}) q^{-(\omega_s, 2 \rho)} p \sigmaPM(\metPM) p \\
& = \del p \otimes \delbar p + (1 - q^{(\alpha_s, \alpha_s)}) q^{-(\alpha_s, \alpha_s)} q^{-(\omega_s, 2 \rho)} q^{2 (\omega_s, \omega_s)} \opS_{1 2 3}^{-1} \opSt_{2 3 4} p \metMP p \\
& + (1 - q^{-(\alpha_s, \alpha_s)}) q^{-(\omega_s, 2 \rho)} p \metMP p.
\end{split}
\]
Here we have also used \eqref{eq:metric-sigma}.
Similarly to what we did above, we compute
\[
\opS_{1 2 3}^{-1} \opSt_{2 3 4} p \metMP p = \opS_{1 2 3}^{-1} \opSt_{2 3 4} p p \metMP
= q^{-2 (\omega_s, \omega_s)} p \metMP p.
\]
Finally we obtain
\[
\begin{split}
\sigmaPM \sigmaMP(\delbar p \otimes \del p)
& = \del p \otimes \delbar p - (1 - q^{-(\alpha_s, \alpha_s)}) q^{-(\omega_s, 2 \rho)} p \metMP p \\
& + (1 - q^{-(\alpha_s, \alpha_s)}) q^{-(\omega_s, 2 \rho)} p \metMP p \\
& = \del p \otimes \delbar p. \qedhere
\end{split}
\]
\end{proof}

\begin{remark}
The relations above can be compared with those satisfied by the braiding on the category of finite-dimensional $\Uqg$-modules.
Formally this is done using Takeuchi's categorical equivalence, see for instance \cite[Section 2.2.8]{heko}.
There is a functor sending the object $\calc \otimes_\subalg \calc$ to the object $\calc / \subalg^+ \calc \otimes_\subalg \calc / \subalg^+ \calc$, where $\subalg^+$ denotes the elements of $\subalg$ killed by the counit.
It can be shown that, under this equivalence, the maps
\[
\sigmaPP, \quad
\sigmaMM, \quad
\sigmaPM, \quad
\sigmaMP,
\]
correspond respectively to the transposes of
\[
\hat{R}_{V, V}, \quad
\hat{R}_{V^*, V^*}, \quad
\hat{R}_{V, V^*}^{-1}, \quad
\hat{R}_{V, V^*},
\]
up to overall scalar multiples (we do not provide the details here).
This can be checked on the generators $\del p$ and $\delbar p$ using \cite[Proposition 3.6 and Proposition 3.7]{heko}.
\end{remark}

\subsection{Presentation calculus}

In this subsection we use the generalized braiding $\sigma$ to give a new presentation of the Heckenberger-Kolb calculus $\calc^\bullet$, paralleling the classical picture.

In order to formulate this more precisely, we need to recall the notion of universal differential calculus corresponding to a FODC, see for instance \cite[Section 2.3.2]{heko}.

\begin{definition}
Let $\Gamma$ be a FODC over an algebra $A$. The \emph{universal differential calculus} corresponding to $\Gamma$ is defined to be the quotient of the tensor algebra $T_A(\Gamma) = \bigoplus_{k = 0}^\infty \Gamma^{\otimes_A k}$ of the $A$-bimodule $\Gamma$ by the ideal generated by the subspace
\[
\left\{ \sum_i \diff a_i \otimes \diff b_i : \sum_i a_i \diff b_i = 0 \right\} \subset \Gamma \otimes_A \Gamma.
\]
\end{definition}

As the name suggests, there is a suitable universal property for $\Gamma_u$, expressing the fact that it is the most general differential calculus having $\Gamma$ as its first-order part.

The differential calculus $\calc^\bullet$ over $\subalg$, introduced by Heckenberger-Kolb in \cite{heko}, is defined to be the universal differential calculus corresponding to the FODC $\calc = \calcP \oplus \calcM$.
The corresponding quadratic relations are also explicitly determined in the cited paper, more precisely in \cite[Proposition 3.6, Proposition 3.7 and Proposition 3.11]{heko}.

We recall them here. The relations among the generators $\del p$ are
\begin{equation}
\label{eq:universal-calcP}
\begin{gathered}
(\braid_{V, V} + q^{(\omega_s, \omega_s) - (\alpha_s, \alpha_s)})_{1 2} (\braid_{V, V^*}^{-1})_{2 3} \del p \wedge \del p = 0, \\
(\braid_{V^*, V^*} - q^{(\omega_s, \omega_s)})_{3 4} (\braid_{V, V^*}^{-1})_{2 3} \del p \wedge \del p = 0.
\end{gathered}
\end{equation}
Similarly, the relations among the generators $\delbar p$ are
\begin{equation}
\label{eq:universal-calcM}
\begin{gathered}
(\braid_{V, V} - q^{(\omega_s, \omega_s)})_{1 2} (\braid_{V, V^*}^{-1})_{2 3} \delbar p \wedge \delbar p = 0, \\
(\braid_{V^*, V^*} + q^{(\omega_s, \omega_s) - (\alpha_s, \alpha_s)})_{3 4} (\braid_{V, V^*}^{-1})_{2 3} \delbar p \wedge \delbar p = 0.
\end{gathered}
\end{equation}
Finally, the mixed relations among $\del p$ and $\delbar p$ are
\begin{equation}
\label{eq:universal-mixed}
\delbar p \wedge \del p = - q^{-(\alpha_s, \alpha_s)} \opT_{1 2 3 4}^{-1} \del p \wedge \delbar p + q^{-(\alpha_s, \alpha_s)} q^{-(\omega_s, 2 \rho)} \EVp_{3 4} \opT^{-1}_{3 4 5 6} p \del p \wedge \delbar p.
\end{equation}
This expression can be rewritten in several ways, as we show in \cref{lem:degree-two-relations}.

We are now ready to formulate our result.

\begin{theorem}
\label{thm:quotient-symmetrizer}
The differential calculus $\calc^\bullet$ over $\subalg$ can be presented as
\[
\calc^\bullet = \talg(\calc) / \langle \mathrm{im} (\id + \sigma) \rangle.
\]
\end{theorem}

\begin{proof}
Since $\calc^\bullet$ is the universal differential calculus corresponding to the FODC $\Omega$, we need to show that the relations \eqref{eq:universal-calcP}, \eqref{eq:universal-calcM} and \eqref{eq:universal-mixed} arise by acting with $\id + \sigma$ on $\calc \otimes_\subalg \calc$ (and quotienting by $\wedge$). Moreover, no further relations should arise this way.

\textit{Relations $\del p$ and $\delbar p$.}
Applying $(\braid_{V, V^*})_{2 3}$ to \eqref{eq:universal-calcP} and \eqref{eq:universal-calcM} we rewrite them as
\[
\begin{gathered}
(\opS_{1 2 3} + q^{(\omega_s, \omega_s) - (\alpha_s, \alpha_s)}) \del p \wedge \del p = 0, \quad
(\opSt_{2 3 4}^{-1} - q^{(\omega_s, \omega_s)}) \del p \wedge \del p = 0, \\
(\opS_{1 2 3} - q^{(\omega_s, \omega_s)}) \delbar p \wedge \delbar p = 0, \quad
(\opSt_{2 3 4}^{-1} + q^{(\omega_s, \omega_s) - (\alpha_s, \alpha_s)}) \delbar p \wedge \delbar p = 0.
\end{gathered}
\]
It follows from \eqref{eq:tensor-algebra-identities} that the second and third relations are identically satisfied in the tensor algebra $\talg(\calc)$.
Therefore consider the elements
\[
(\opS_{1 2 3} + q^{(\omega_s, \omega_s) - (\alpha_s, \alpha_s)}) \del p \otimes \del p, \quad
(\opSt_{2 3 4}^{-1} + q^{(\omega_s, \omega_s) - (\alpha_s, \alpha_s)}) \delbar p \otimes \delbar p.
\]
It is easy to check that these correspond (up to a scalar) with
\[
(\sigmaPP + \id)(\del p \otimes \del p), \quad
(\sigmaMM + \id)(\delbar p \otimes \delbar p).
\]
Indeed, these follow from the expressions $\sigmaPP(\del p \otimes \del p) = q^{(\alpha_s, \alpha_s) - (\omega_s, \omega_s)} \opS_{1 2 3} \del p \otimes \del p$ and $\sigmaMM(\delbar p \otimes \delbar p) = q^{(\omega_s, \omega_s) - (\alpha_s, \alpha_s)} \opSt_{2 3 4} \delbar p \otimes \delbar p$.

\textit{Mixed relations.}
Corresponding to the relation \eqref{eq:universal-mixed} we consider the element
\[
A = \delbar p \otimes \del p + q^{-(\alpha_s, \alpha_s)} \opT_{1 2 3 4}^{-1} \del p \otimes \delbar p - q^{-(\alpha_s, \alpha_s)} q^{-(\omega_s, 2 \rho)} \EVp_{3 4} \opT^{-1}_{3 4 5 6} p \del p \otimes \delbar p.
\]
We write $A_1 = \opT_{1 2 3 4}^{-1} \del p \otimes \delbar p$. Taking into account that $\opT_{1 2 3 4}^{-1} = \opS_{1 2 3}^{-1} \opSt_{2 3 4}^{-1}$ and using the quadratic relation \eqref{eq:S-quadratic-relations} for $\opSt_{2 3 4}^{-1}$, we rewrite this term as
\[
A_1 = q^{2 (\omega_s, \omega_s) - (\alpha_s, \alpha_s)} \opS_{1 2 3}^{-1} \opSt_{2 3 4} \del p \otimes \delbar p + (1 - q^{-(\alpha_s, \alpha_s)}) q^{(\omega_s, \omega_s)} \opS_{1 2 3}^{-1} \del p \otimes \delbar p.
\]
On the other hand consider $A_2 = \EVp_{3 4} \opT^{-1}_{3 4 5 6} p \del p \otimes \delbar p$. Using \cref{lem:evaluation-T} we get
\[
A_2 = - (1 - q^{(\alpha_s, \alpha_s)}) q^{(\omega_s, 2 \rho)} \EV_{4 5} p \del p \otimes \delbar p + (1 - q^{(\alpha_s, \alpha_s)}) p \metPM p.
\]
Moreover, using \eqref{eq:right-module-relations} and \eqref{eq:evaluations-calculi} we rewrite
\[
\EV_{4 5} p \del p \otimes \delbar p = q^{(\omega_s, \omega_s) - (\alpha_s, \alpha_s)} \EV_{4 5} \opS_{1 2 3}^{-1} \del p p \otimes \delbar p
= q^{(\omega_s, \omega_s) - (\alpha_s, \alpha_s)} \opS_{1 2 3}^{-1} \del p \otimes \delbar p.
\]
Taking these identities into account we obtain
\[
\begin{split}
A & = \delbar p \otimes \del p + q^{-(\alpha_s, \alpha_s)} A_1 - q^{-(\alpha_s, \alpha_s)} q^{-(\omega_s, 2 \rho)} A_2 \\
& = \delbar p \otimes \del p + q^{2 (\omega_s, \omega_s) - 2 (\alpha_s, \alpha_s)} \opS_{1 2 3}^{-1} \opSt_{2 3 4} \del p \otimes \delbar p + (1 - q^{-(\alpha_s, \alpha_s)}) q^{-(\omega_s, 2 \rho)} p \metPM p \\
& = (\id + \sigma)(\delbar p \otimes \del p).
\end{split}
\]

\textit{No further relations.}
Finally we show that no additional relations arise by acting with $\id + \sigma$ on $\calcP \otimes_\subalg \calcM$.
Considering the element $B = (\id + \sigma)(\del p \otimes \delbar p)$, our goal is to show that $\wedge(B) = 0$, which means that this relation is already satisfied in $\calc^\bullet$.

From the proof of \cite[Lemma 8.2]{paper-projective} we have the alternative expression
\[
\begin{split}
\sigmaPM(\del p \otimes \delbar p) & = q^{(\alpha_s, \alpha_s)} \opT_{1 2 3 4} \delbar p \otimes \del p
- (1 - q^{(\alpha_s, \alpha_s)}) q^{(\alpha_s, \alpha_s)} \EV_{2 3} \opT_{1 2 3 4} \delbar p \otimes \del p p \\
& + (1 - q^{(\alpha_s, \alpha_s)}) q^{(\alpha_s, \alpha_s)} q^{-(\omega_s, 2 \rho)} p \metMP p.
\end{split}
\]
Using the identity from \cref{lem:evaluation-T} this can be rewritten as
\[
\sigmaPM(\del p \otimes \delbar p) = q^{(\alpha_s, \alpha_s)} \opT_{1 2 3 4} \delbar p \otimes \del p
- q^{(\alpha_s, \alpha_s)} q^{-(\omega_s, 2 \rho)} \EVp_{3 4} \opT_{1 2 3 4} \delbar p \otimes \del p p.
\]
Plugging this into $B = (\id + \sigma)(\del p \otimes \delbar p)$ we obtain
\[
\wedge(B) = \del p \wedge \delbar p + q^{(\alpha_s, \alpha_s)} \opT_{1 2 3 4} \delbar p \wedge \del p
- q^{(\alpha_s, \alpha_s)} q^{-(\omega_s, 2 \rho)} \EVp_{3 4} \opT_{1 2 3 4} \delbar p \wedge \del p p.
\]
But we have $\wedge(B) = 0$, as shown in \cref{lem:degree-two-relations}.
\end{proof}

This presentation immediately leads to the following property.

\begin{corollary}
\label{cor:compatibility-wedge}
We have $\wedge \circ (\id + \sigma) = 0$.
\end{corollary}

\begin{remark}
The property $\wedge \circ (\id + \sigma) = 0$ is called \emph{torsion compatibility} in \cite[Section 8.1]{quantum-book}, in the context of a bimodule connection with generalized braiding $\sigma$.
When this holds, the strong form of metric compatibility (requiring a bimodule connection) implies the weak form (vanishing cotorsion), see \cite{paper-projective} for a discussion of these properties.
\end{remark}

\section{Additional results}
\label{sec:additional-results}

In this section we derive various additional results within our geometrical setting for quantum projective spaces.
We introduce a one-parameter family of maps $\splitmap: \calc^2 \to \calc \otimes_\subalg \calc$ splitting the wedge product, which we are going to use later to define the Ricci tensor.
We compute the quantum metric dimension of the differential calculus, defined as the composition of the inverse metric $\imet$ and the quantum metric $\met$.
Finally we prove a symmetry property for the inverse metric, which involves the generalized braiding $\sigma$.

\subsection{Splitting map}

We look for a splitting of the wedge product $\wedge: \calc \otimes_\subalg \calc \to \calc^2$, that is a $\subalg$-bimodule map $\splitmap: \calc^2 \to \calc \otimes_\subalg \calc$ such that $\wedge \circ \splitmap = \id$.
It is technically more convenient to first define $\splitmap$ as a map $\calc \otimes_\subalg \calc \to \calc \otimes_\subalg \calc$, and then ask that it descends to the space of two-forms $\calc^2$.
In the classical case this map is given by antisymmetrization, that is $x \otimes y \mapsto \frac{1}{2} (x \otimes y - y \otimes x)$.
Keeping this in mind, we look for a map $\splitmap$ which is a linear combination of $\id$ and the generalized braiding $\sigma$ in each component $\calc_a \otimes \calc_b$.
More concretely, we assume that
\[
\splitmap |_{\calc_a \otimes \calc_b} = c_{a b} \id - c_{a b}^\prime \sigma_{a b}, \quad
a, b \in \{ +, - \}.
\]
Note that any such $\splitmap$ is a $\subalg$-bimodule map, since this is true for each component $\sigma_{a b}$ of the generalized braiding.
Within this setting, we have the following result.

\begin{proposition}
\label{prop:split-map}
Let $\splitmap: \calc \otimes_\subalg \calc \to \calc \otimes_\subalg \calc$ be as above.
Suppose that:
\begin{enumerate}
\item it descends to a map $\calc^2 \to \calc \otimes_\subalg \calc$,
\item we have $\wedge \circ \splitmap = \wedge$.
\end{enumerate}
Then it must be of the form
\[
\begin{split}
\splitmap(\del p \otimes \del p) & = \frac{1}{1 + q^{-(\alpha_s, \alpha_s)}} (\id - q^{-(\alpha_s, \alpha_s)} \sigmaPP) (\del p \otimes \del p), \\
\splitmap(\delbar p \otimes \delbar p) & = \frac{1}{1 + q^{(\alpha_s, \alpha_s)}} (\id - q^{(\alpha_s, \alpha_s)} \sigmaPP) (\delbar p \otimes \delbar p), \\
\splitmap(\del p \otimes \delbar p) & = (c_{+ -} \id - c_{- +} \sigmaPM) (\del p \otimes \delbar p), \\
\splitmap(\delbar p \otimes \del p) & = (c_{- +} \id - c_{+ -} \sigmaMP) (\delbar p \otimes \del p).
\end{split}
\]
In addition we must have $c_{+ -} + c_{- +} = 1$.
\end{proposition}

\begin{proof}
We are going to show that requiring the conditions (1) and (2) to hold for the map $\splitmap |_{\calc_a \otimes \calc_b} = c_{a b} \id - c_{a b}^\prime \sigma_{a b}$ fixes all the coefficients $c_{a b}$ but one.

(1) According to \cref{thm:quotient-symmetrizer}, the differential calculus $\calc^\bullet$ is the quotient of the tensor algebra $\talg(\calc)$ by the ideal generated by $\mathrm{im}(\id + \sigma)$.
Hence to show that $\splitmap$ descends to a map $\calc^2 \to \calc \otimes_\subalg \calc$ it suffices to show that $\splitmap \circ (\id + \sigma) = 0$.
We can determine when this holds using the quadratic relations for the components $\sigma_{a b}$ from \cref{prop:quadratic-relations}.

First, using $\sigmaPP^2 = q^{(\alpha_s, \alpha_s)} - (1 - q^{(\alpha_s, \alpha_s)}) \sigmaPP$ we compute
\[
\splitmap (\id + \sigma) (\del p \otimes \del p) = (c_{+ +} - q^{(\alpha_s, \alpha_s)} c_{+ +}^\prime) (\id + \sigmaPP) (\del p \otimes \del p).
\]
For this to vanish we must have $c_{+ +}^\prime = q^{-(\alpha_s, \alpha_s)} c_{+ +}$.

Similarly, using $\sigmaMM^2 = q^{-(\alpha_s, \alpha_s)} - (1 - q^{-(\alpha_s, \alpha_s)}) \sigmaMM$ we compute
\[
\splitmap (\id + \sigma) (\delbar p \otimes \delbar p) = (c_{- -} - q^{-(\alpha_s, \alpha_s)} c_{- -}^\prime) (\id + \sigmaMM) (\delbar p \otimes \delbar p).
\]
For this to vanish we must have $c_{- -}^\prime = q^{(\alpha_s, \alpha_s)} c_{- -}$.

Next, for the mixed case of $\calcP \otimes \calcM$ we compute
\[
\begin{split}
\splitmap (\id + \sigma) (\del p \otimes \delbar p) & = (c_{+ -} \id - c_{+ -}^\prime \sigmaPM) (\del p \otimes \delbar p) + (c_{- +} \id - c_{- +}^\prime \sigmaMP) \sigmaPM (\del p \otimes \delbar p) \\
& = (c_{+ -} - c_{- +}^\prime) (\del p \otimes \delbar p) + (c_{- +} - c_{+ -}^\prime) \sigmaPM (\del p \otimes \delbar p).
\end{split}
\]
For this to vanish we must have $c_{- +}^\prime = c_{+ -}$ and $c_{+ -}^\prime = c_{- +}$.
The computation for the case of $\calcM \otimes \calcP$ is very similar and leads to the same conditions.

Hence, after imposing the condition (1), the map $\splitmap$ looks as follows
\[
\begin{split}
\splitmap(\del p \otimes \del p) & = c_{+ +} (\id - q^{-(\alpha_s, \alpha_s)} \sigmaPP) (\del p \otimes \del p), \\
\splitmap(\delbar p \otimes \delbar p) & = c_{- -} (\id - q^{(\alpha_s, \alpha_s)} \sigmaPP) (\delbar p \otimes \delbar p), \\
\splitmap(\del p \otimes \delbar p) & = (c_{+ -} \id - c_{- +} \sigmaPM) (\del p \otimes \delbar p), \\
\splitmap(\delbar p \otimes \del p) & = (c_{- +} \id - c_{+ -} \sigmaMP) (\delbar p \otimes \del p).
\end{split}
\]

(2) Now we require the condition $\wedge \circ \splitmap = \wedge$.
To check this condition we can use the property $\wedge \circ \sigma = - \wedge$ from \cref{cor:compatibility-wedge}.
For $\del p \otimes \del p$ we compute
\[
\wedge \circ \splitmap(\del p \otimes \del p) = c_{+ +} (1 + q^{-(\alpha_s, \alpha_s)}) \del p \wedge \del p.
\]
Hence we must have $c_{+ +} = (1 + q^{-(\alpha_s, \alpha_s)})^{-1}$.
A similar computation for $\delbar p \otimes \delbar p$ leads to $c_{- -} = (1 + q^{(\alpha_s, \alpha_s)})^{-1}$.
For the mixed case it leads to $c_{+ -} + c_{- +} = 1$.

We conclude that the conditions (1) and (2) fix all but one of the coefficients $c_{a b}$ and $c_{a b}^\prime$.
Using the explicit relations we obtain the map as in the statement.
\end{proof}

It follows that we have a well-defined map $\splitmap: \calc^2 \to \calc \otimes_\subalg \calc$ (denoted by the same symbol) such that $\wedge \circ \splitmap = \id$.
This leads to the following result.

\begin{corollary}
We have a split exact sequence
\[
0 \to \ker(\wedge) \to \calc \otimes_\subalg \calc \xrightarrow{\wedge} \calc^2 \to 0.
\]
\end{corollary}

\begin{proof}
This sequence is exact for any differential calculus.
Splitting is equivalent to the existence of a map $\splitmap: \calc^2 \to \calc \otimes_\subalg \calc$ such that $\wedge \circ \splitmap = \id$, which is what we obtained.
\end{proof}

\begin{remark}
The splitting of the wedge product is a natural condition to require for a differential calculus.
For instance, in \cite{koszul-formula} it is shown that it implies the existence of a torsion-free connection, under some further assumptions.
However, these additional assumptions are not satisfied by the differential calculus $\calc^\bullet$ we consider here.
\end{remark}

\subsection{Quantum metric dimension}

We are now going to compute the expression $\imet(\met)$, where $\met$ is the quantum metric and $\imet$ is the inverse metric.
In the classical setting it corresponds to the trace of the metric tensor $\met$, hence the dimension of the underlying space.
For this reason $\imet(\met)$ is called the \emph{quantum metric dimension} in \cite[Chapter 8]{quantum-book}.
A priori it is an element of the algebra $\subalg$, but in our case it turns out to be a scalar.
We are going to use the quantum metric dimension in the computation of the Ricci tensor.

First we recall the notion of \emph{quantum dimension} of a simple module $V$, given by
\[
\qdim(V) := \mathrm{Tr}_V(K_{2 \rho}) = \sum_i q^{(\lambda_i, 2 \rho)}.
\]
We can derive a more convenient expression for $\qdim(V)$ in our case.
We employ the standard notion of $q$-numbers, defined by $[x]_q := \frac{q^x - q^{-x}}{q - q^{-1}}$.

\begin{lemma}
\label{lem:quantum-dimension}
Consider the $U_q(\mathfrak{sl}_{r + 1})$-module $V = V(\omega_s)$ with $s = 1$ or $s = r$. Then
\[
\qdim(V) = [r + 1]_q = \frac{q^{(\omega_s, 2 \rho) + (\alpha_s, \alpha_s)/2} - q^{-(\omega_s, 2 \rho) - (\alpha_s, \alpha_s)/2}}{q^{(\alpha_s, \alpha_s)/2} - q^{-(\alpha_s, \alpha_s)/2}}.
\]
\end{lemma}

\begin{proof}
For any simple $\Uqg$-module $V(\lambda)$, the quantum Weyl dimension formula gives
\[
\qdim(V(\lambda)) = \prod_{\alpha > 0} \frac{[(\lambda + \rho, \alpha)]_q}{[(\rho, \alpha)]_q}.
\]
Here the product is over the positive roots of $\lieg$.

We use the above formula with $\lieg = \mathfrak{sl}_{r + 1}$ and $\lambda = \omega_1$ (the case $\lambda = \omega_r$ is identical).
Let us enumerate the positive roots of $\mathfrak{sl}_{r + 1}$ by $\alpha_{i j} = \sum_{k = i}^j \alpha_k$ with $1 \leq i \leq j \leq r$, where $\{\alpha_i\}_{i = 1}^r$ are the simple roots. Then $(\omega_1, \alpha_{i j}) = 0$ for $i > 1$ and $(\omega_1, \alpha_{1 j}) = 1$.
Moreover we have $(\rho, \alpha_{1 j}) = j$, since $\rho$ is the sum of all fundamental weights. Then we obtain
\[
\qdim(V) = \prod_{j = 1}^r \frac{[(\omega_1 + \rho, \alpha_{1 j})]_q}{[(\rho, \alpha_{1 j})]_q}
= \prod_{j = 1}^r \frac{[j + 1]_q}{[j]_q} = [r + 1]_q.
\]
The second expression for $\qdim(V)$ follows easily from this one, using the fact that for $\mathfrak{sl}_{r + 1}$ we have the values $(\alpha_s, \alpha_s) = 2$ and $(\omega_s, 2 \rho) = r$.
\end{proof}

We now compute the quantum metric dimension of the differential calculus $\calc^\bullet$.

\begin{proposition}
\label{prop:quantum-trace}
For the quantum metric $\met = \metPM + \metMP$ we have
\[
\begin{split}
\imet(\metPM) & = q^{-(\omega_s, 2 \rho)} \qdim(V) - 1, \\
\imet(\metMP) & = q^{(\omega_s, 2 \rho)} \qdim(V) - 1.
\end{split}
\]
Moreover the quantum metric dimension can be written as
\[
\imet(\met) = (q^{r + 1} + q^{-(r + 1)}) [r]_q.
\]
\end{proposition}

\begin{proof}
\textit{The term $(\cdot, \cdot)(\metPM)$.}
Using the inverse metric \eqref{eq:inverse-metric} we get
\[
\imet(\metPM) = q^{(\alpha_s, \alpha_s)} q^{-(\omega_s, \omega_s + 2 \rho)} \EVp_{1 2} \EV_{2 3} \opS_{1 2 3} \CV_3 p - q^{(\alpha_s, \alpha_s)} q^{-(\omega_s, 2 \rho)} \EVp_{1 2} \EV_{2 3} p p.
\]
For the second term we have $\EVp_{1 2} \EV_{2 3} p p = \EVp_{1 2} p = q^{(\omega_s, 2 \rho)}$.
For the first term, using the quadratic relation \eqref{eq:S-quadratic-relations} for $\opS_{1 2 3}$ and the duality relation from \eqref{eq:duality}, we compute
\[
\begin{split}
\EVp_{1 2} \EV_{2 3} \opS_{1 2 3} \CV_3 p
& = q^{2 (\omega_s, \omega_s) - (\alpha_s, \alpha_s)} \EVp_{1 2} \EV_{2 3} \opS_{1 2 3}^{-1} \CV_3 p + (1 - q^{-(\alpha_s, \alpha_s)}) q^{(\omega_s, \omega_s)} \EVp_{1 2} \EV_{2 3} \CV_3 p \\
& = q^{(\omega_s, \omega_s) - (\alpha_s, \alpha_s)} q^{-(\omega_s, 2 \rho)} \EVp_{1 2} \EVp_{1 2} \CV_3 p + (1 - q^{-(\alpha_s, \alpha_s)}) q^{(\omega_s, \omega_s)} \EVp_{1 2} p \\
& = q^{(\omega_s, \omega_s) - (\alpha_s, \alpha_s)} q^{-(\omega_s, 2 \rho)} \EVp_{1 2} \CV_1 \EVp_{1 2} p + (1 - q^{-(\alpha_s, \alpha_s)}) q^{(\omega_s, \omega_s + 2 \rho)} \\
& = q^{(\omega_s, \omega_s) - (\alpha_s, \alpha_s)} \EVp_{1 2} \CV_1 + (1 - q^{-(\alpha_s, \alpha_s)}) q^{(\omega_s, \omega_s + 2 \rho)}.
\end{split}
\]
The expression $\EVp_{1 2} \CV_1$ corresponds to the quantum dimension, since
\[
\EVp_{1 2} \CV_1 = \sum_i \EVp_{1 2}(v_i \otimes f^i) = \sum_i q^{(2 \rho, \lambda_i)} = \qdim(V).
\]
Therefore we obtain the result
\[
\begin{split}
\imet(\metPM) & = q^{-(\omega_s, 2 \rho)} \qdim(V) + (q^{(\alpha_s, \alpha_s)} - 1) - q^{(\alpha_s, \alpha_s)} \\
& = q^{-(\omega_s, 2 \rho)} \qdim(V) - 1.
\end{split}
\]

\textit{The term $(\cdot, \cdot)(\metMP)$.}
Similarly we have
\[
\imet(\metMP) = \EVp_{1 2} \EV_{2 3} \CVp_2 p - q^{-(\omega_s, 2 \rho)} \EVp_{1 2} \EV_{2 3} p p.
\]
We have $\EVp_{1 2} \EV_{2 3} p p = q^{(\omega_s, 2 \rho)}$, as above.
For the first term, using \cref{lem:categorical-identities} we get
\[
\EVp_{1 2} \EV_{2 3} \CVp_2 p = q^{(\omega_s, \omega_s + 2 \rho)} \EVp_{1 2} \EV_{2 3} (\braid_{V, V^*})_{2 3} \CV_2 p \\
= \EVp_{1 2} \EVp_{2 3} \CV_2 p.
\]
Using $\EVp_{2 3} \CV_2 = \qdim(V)$, as above, we find that
\[
\EVp_{1 2} \EV_{2 3} \CVp_2 p = \qdim(V) \EVp_{1 2} p = q^{(\omega_s, 2 \rho)} \qdim(V).
\]
Therefore we obtain the result
\[
\imet(\metMP) = q^{(\omega_s, 2 \rho)} \qdim(V) - 1.
\]

\textit{Quantum metric dimension.}
We have $\imet(\met) = (q^{(\omega_s, 2 \rho)} + q^{-(\omega_s, 2 \rho)}) \qdim(V) - 2$.
Using \cref{lem:quantum-dimension} and $(\omega_s, 2 \rho) = r$, plus the identity $[r + 1]_q = q [r]_q + q^{-r}$, we check that
\[
(q^{(\omega_s, 2 \rho)} + q^{-(\omega_s, 2 \rho)}) \qdim(V) - 2 = (q^{r + 1} + q^{-(r + 1)}) [r]_q.
\]
This gives an expression which is closer to the classical one.
\end{proof}

\begin{remark}
The quantum metric dimension reduces to $(\cdot, \cdot)(\met) = 2 r$ in the classical limit $q \to 1$, which coincides the dimension of $\bbC P^r$ as a real manifold.
This can also be checked using the general fact that $\qdim(V)$ reduces to $\dim(V)$ in the classical limit.
\end{remark}

In the following we are going to write the quantum metric dimension as
\[
\Trq(\met) := (\cdot, \cdot)(\met).
\]
We employ similar notations for $\metPM$ and $\metMP$.

\subsection{Symmetry inverse metric}

We have seen that the quantum metric $\met$ is \emph{symmetric}, in the sense that $\wedge(\met) = 0$.
What about the inverse metric $\imet$?
We now show that it satisfies a certain twisted symmetry which involves the generalized braiding $\sigma$.

\begin{proposition}
\label{prop:inverse-metric-symmetry}
The inverse metric $\imet: \calc \otimes_\subalg \calc \to \subalg$ satisfies
\[
\begin{split}
\imet \circ \sigma (\del p \otimes \delbar p) & = q^{(\alpha_s, \alpha_s)} q^{2 (\omega_s, 2 \rho)} \imet (\del p \otimes \delbar p), \\
\imet \circ \sigma (\delbar p \otimes \del p) & = q^{-(\alpha_s, \alpha_s)} q^{-2 (\omega_s, 2 \rho)} \imet (\delbar p \otimes \del p).
\end{split}
\]
\end{proposition}

\begin{proof}
\textit{First identity.}
Using \eqref{eq:inverse-metric} and \cref{prop:quantum-trace} we obtain
\[
\begin{split}
\imet \sigma (\del p \otimes \delbar p)
& = q^{2 (\alpha_s, \alpha_s) - 2 (\omega_s, \omega_s)} \opS_{1 2 3} \opSt_{2 3 4}^{-1} (\delbar p, \del p) + (1 - q^{(\alpha_s, \alpha_s)}) q^{(\alpha_s, \alpha_s)} q^{-(\omega_s, 2 \rho)} \Trq(\metMP) p p \\
& = q^{2 (\alpha_s, \alpha_s) - 2 (\omega_s, \omega_s)} \opS_{1 2 3} \opSt_{2 3 4}^{-1} \CVp_2 p - q^{2 (\alpha_s, \alpha_s) - 2 (\omega_s, \omega_s)} q^{-(\omega_s, 2 \rho)} \opS_{1 2 3} \opSt_{2 3 4}^{-1} p p \\
& + (1 - q^{(\alpha_s, \alpha_s)}) q^{(\alpha_s, \alpha_s)} q^{-(\omega_s, 2 \rho)} (q^{(\omega_s, 2 \rho)} \qdim(V) - 1) p p.
\end{split}
\]
Using \cref{lem:categorical-identities} and naturality of the coevaluation $\CV$ as in \eqref{eq:coevaluations}, we get
\[
\opSt_{2 3 4}^{-1} \CVp_2 = q^{(\omega_s, \omega_s + 2 \rho)} \opSt_{2 3 4}^{-1} (\braid_{V, V^*})_{2 3} \CV_2
= q^{(\omega_s, \omega_s + 2 \rho)} (\braid_{V, V^*})_{2 3} (\braid_{V^*, V^*})_{3 4} \CV_2 = q^{(\omega_s, \omega_s + 2 \rho)} \CV_3.
\]
We also have $\opS_{1 2 3} \opSt_{2 3 4}^{-1} p p = q^{2 (\omega_s, \omega_s)} p p$ from \eqref{eq:algebra-relations}. Then we obtain
\[
\begin{split}
\imet \sigma (\del p \otimes \delbar p)
& = q^{2 (\alpha_s, \alpha_s) - (\omega_s, \omega_s)} q^{(\omega_s, 2 \rho)} \opS_{1 2 3} \CV_3 p - q^{2 (\alpha_s, \alpha_s)} q^{-(\omega_s, 2 \rho)} p p \\
& + (1 - q^{(\alpha_s, \alpha_s)}) q^{(\alpha_s, \alpha_s)} q^{-(\omega_s, 2 \rho)} (q^{(\omega_s, 2 \rho)} \qdim(V) - 1) p p.
\end{split}
\]
After some simplifications this can be rewritten as
\[
\begin{split}
\imet \sigma (\del p \otimes \delbar p)
& = q^{2 (\alpha_s, \alpha_s) - (\omega_s, \omega_s)} q^{(\omega_s, 2 \rho)} \opS_{1 2 3} \CV_3 p \\
& - \left( q^{(\alpha_s, \alpha_s)} q^{-(\omega_s, 2 \rho)} - (1 - q^{(\alpha_s, \alpha_s)}) q^{(\alpha_s, \alpha_s)} \qdim(V) \right) p p.
\end{split}
\]
Next, using \cref{lem:quantum-dimension} we can easily verify the identity
\[
q^{(\alpha_s, \alpha_s)} q^{-(\omega_s, 2 \rho)} - (1 - q^{(\alpha_s, \alpha_s)}) q^{(\alpha_s, \alpha_s)} \qdim(V)
= q^{2 (\alpha_s, \alpha_s)} q^{(\omega_s, 2 \rho)}.
\]
Finally taking into account \eqref{eq:inverse-metric} we obtain
\[
\begin{split}
\imet \sigma (\del p \otimes \delbar p)
& = q^{(\alpha_s, \alpha_s)} q^{2(\omega_s, 2 \rho)} (q^{(\alpha_s, \alpha_s)} q^{-(\omega_s, \omega_s + 2 \rho)} \opS_{1 2 3} \CV_3 p - q^{(\alpha_s, \alpha_s)} q^{-(\omega_s, 2 \rho)} p p) \\
& = q^{(\alpha_s, \alpha_s)} q^{2 (\omega_s, 2 \rho)} \imet (\del p \otimes \delbar p).
\end{split}
\]

\textit{Second identity.}
Similarly to the first case, we compute
\[
\begin{split}
\imet \sigma (\delbar p \otimes \del p)
& = q^{2 (\omega_s, \omega_s) - 2 (\alpha_s, \alpha_s)} \opS_{1 2 3}^{-1} \opSt_{2 3 4} (\del p, \delbar p) + (1 - q^{-(\alpha_s, \alpha_s)}) q^{-(\omega_s, 2 \rho)} \Trq(\metPM) p p \\
& = q^{(\omega_s, \omega_s) - (\alpha_s, \alpha_s)} q^{-(\omega_s, 2 \rho)} \opSt_{2 3 4} \CV_3 p - q^{2 (\omega_s, \omega_s) - (\alpha_s, \alpha_s)} q^{-(\omega_s, 2 \rho)} \opS_{1 2 3}^{-1} \opSt_{2 3 4} p p \\
& + (1 - q^{-(\alpha_s, \alpha_s)}) q^{-(\omega_s, 2 \rho)} (q^{-(\omega_s, 2 \rho)} \qdim(V) - 1) p p.
\end{split}
\]
We have seen in the derivation of the first identity that $\opSt_{2 3 4}^{-1} \CVp_2 = q^{(\omega_s, \omega_s + 2 \rho)} \CV_3$.
We also have $\opS_{1 2 3}^{-1} \opSt_{2 3 4} p p = q^{-2 (\omega_s, \omega_s)} p p$ using \eqref{eq:algebra-relations}. Then we get
\[
\begin{split}
\imet \sigma (\delbar p \otimes \del p)
& = q^{-(\alpha_s, \alpha_s)} q^{-2 (\omega_s, 2 \rho)} \CVp_2 p - q^{-(\alpha_s, \alpha_s)} q^{-(\omega_s, 2 \rho)} p p \\
& + (1 - q^{-(\alpha_s, \alpha_s)}) q^{-(\omega_s, 2 \rho)} (q^{-(\omega_s, 2 \rho)} \qdim(V) - 1) p p.
\end{split}
\]
After some simplifications we get
\[
\begin{split}
\imet \sigma (\delbar p \otimes \del p)
& = q^{-(\alpha_s, \alpha_s)} q^{-2 (\omega_s, 2 \rho)} \CVp_2 p \\
& - \left( q^{-(\omega_s, 2 \rho)} - (1 - q^{-(\alpha_s, \alpha_s)}) q^{-2 (\omega_s, 2 \rho)} \qdim(V) \right) p p.
\end{split}
\]
Now using \cref{lem:quantum-dimension} we can easily verify the identity
\[
q^{-(\omega_s, 2 \rho)} - (1 - q^{-(\alpha_s, \alpha_s)}) q^{-2 (\omega_s, 2 \rho)} \qdim(V) = q^{-(\alpha_s, \alpha_s)} q^{-3 (\omega_s, 2 \rho)}.
\]
Finally taking into account \eqref{eq:inverse-metric} we obtain
\[
\begin{split}
\imet \sigma (\delbar p \otimes \del p)
& = q^{-(\alpha_s, \alpha_s)} q^{-2 (\omega_s, 2 \rho)} (\CVp_2 p - q^{-(\omega_s, 2 \rho)} p p) \\
& = q^{-(\alpha_s, \alpha_s)} q^{-2 (\omega_s, 2 \rho)} \imet (\delbar p \otimes \del p). \qedhere
\end{split}
\]
\end{proof}

\begin{remark}
This result gives a twisted symmetry for the inverse metric $\imet$, reducing to the usual symmetry in the classical case.
It can be written in a nicer way by defining
\[
\tilde{\sigma} =
\begin{cases}
\sigma & \textrm{on } \calcP \otimes \calcP \\
\sigma & \textrm{on } \calcM \otimes \calcM \\
q^{-(\alpha_s, \alpha_s)} q^{-2 (\omega_s, 2 \rho)} \sigma & \textrm{on } \calcP \otimes \calcM \\
q^{(\alpha_s, \alpha_s)} q^{2 (\omega_s, 2 \rho)} \sigma & \textrm{on } \calcM \otimes \calcP
\end{cases}
.
\]
Then $\tilde{\sigma}$ is a $\subalg$-bimodule map $\calc \otimes_\subalg \calc \to \calc \otimes_\subalg \calc$ and we have $\imet \circ \tilde{\sigma} = \imet$.
\end{remark}

\section{Riemann tensor}
\label{sec:riemann-tensor}

In this section we compute the Riemann tensor corresponding to the quantum Levi-Civita connection $\nabla$.
We also show that it is a $\subalg$-bimodule map, a property that is not guaranteed to hold in general by its definition.
Finally we show various consequences of this fact, namely the extendability of the generalized braiding $\sigma$, an antisymmetry identity for the Riemann tensor and also a version of the braid equation satisfied by $\sigma$.

\subsection{Definition and computation}

We define the Riemann tensor as the curvature of the connection $\conn: \calc \to \calc \otimes_\subalg \calc$.
Indeed, in the classical case this corresponds to the usual definition up to an overall constant (see for instance \cite[Example 3.29]{quantum-book}).

\begin{definition}
The \emph{Riemann tensor} is the map $\curv: \calc \to \calc^2 \otimes_\subalg \calc$ defined by
\[
\curv := (\diff \otimes \id - (\wedge \otimes \id) \circ (\id \otimes \conn)) \circ \conn.
\]
\end{definition}

We point out that, in general, this is a left module map but not necessarily a right module map.
In our case, however, we are going to show that it is a $\subalg$-bimodule map.

The rest of this subsection is devoted to explicitly computing $\curv$.

\begin{lemma}
\label{lem:riemann-minus}
We have
\[
\curv(\delbar p) = - \EV_{2 3} \EV_{2 3} \delbar p \wedge \del p \otimes \delbar p - q^{-(\omega_s, 2 \rho)} \delbar p \wedge \metPM. \qedhere
\]
\end{lemma}

\begin{proof}
According to \eqref{eq:connection}, the quantum Levi-Civita connection acts on $\delbar p$ by
\[
\conn(\delbar p) = \EV_{2 3} \del p \otimes \delbar p - q^{-(\omega_s, 2 \rho)} p \metPM.
\]
First, using $(\diff \otimes \id)(\metPM) = 0$ from \eqref{eq:kahler-property} we compute
\[
(\diff \otimes \id) \conn(\delbar p) = \EV_{2 3} \delbar \del p \otimes \delbar p - q^{-(\omega_s, 2 \rho)} \diff p \wedge \metPM.
\]
We have the identity (see for instance \cite[Lemma C.6]{paper-projective})
\[
\del \delbar p = \EV_{2 3} \del p \wedge \delbar p + \EV_{2 3} \delbar p \wedge \del p.
\]
Then using $\delbar \del = - \del \delbar$ we can rewrite
\[
\EV_{2 3} \delbar \del p \otimes \delbar p
= - \EV_{2 3} \EV_{2 3} \del p \wedge \delbar p \otimes \delbar p - \EV_{2 3} \EV_{2 3} \delbar p \wedge \del p \otimes \delbar p.
\]
The first term vanishes, since $\EV_{2 3} \EV_{2 3} = \EV_{2 3} \EV_{4 5}$ and $\EV_{2 3} \delbar p \otimes \delbar p = 0$ from \eqref{eq:tensor-algebra-evaluations}. Then
\[
(\diff \otimes \id) \conn(\delbar p) = - \EV_{2 3} \EV_{2 3} \delbar p \wedge \del p \otimes \delbar p - q^{-(\omega_s, 2 \rho)} \diff p \wedge \metPM.
\]

Similarly, using $(\id \otimes \conn)(\metPM) = 0$ from \eqref{eq:metric-compatibility} we obtain
\[
\begin{split}
(\id \otimes \conn) \conn(\delbar p) = \EV_{2 3} \EV_{4 5} \del p \otimes \del p \otimes \delbar p - q^{-(\omega_s, 2 \rho)} \EV_{2 3} \del p \otimes p \metPM.
\end{split}
\]
The first term vanishes since $\EV_{2 3} \EV_{4 5} = \EV_{2 3} \EV_{2 3}$ and $\EV_{2 3} \del p \otimes \del p = 0$ from \eqref{eq:tensor-algebra-evaluations}. Then
\[
(\id \otimes \conn) \conn(\delbar p) = - q^{-(\omega_s, 2 \rho)} \EV_{2 3} \del p p \otimes \metPM = - q^{-(\omega_s, 2 \rho)} \del p \otimes \metPM.
\]

Putting the two terms together and using $\diff = \del + \delbar$ we obtain
\[
\begin{split}
\curv(\delbar p) & = (\diff \otimes \id) \conn(\delbar p) - (\wedge \otimes \id) (\id \otimes \conn) \conn(\delbar p) \\
& = - \EV_{2 3} \EV_{2 3} \delbar p \wedge \del p \otimes \delbar p - q^{-(\omega_s, 2 \rho)} \delbar p \wedge \metPM. \qedhere
\end{split}
\]
\end{proof}

We proceed by computing $\curv(\del p)$, which is technically more challenging. We derive two different expressions for it.
The second one, which is seemingly more complicated, is going to be used to simplify things in the computation of the Ricci tensor.

\begin{lemma}
\label{lem:riemann-plus}
We have
\[
\begin{split}
\curv(\del p)
& = q^{(\alpha_s, \alpha_s)} \EV_{2 3} \opT_{1 2 3 4} \EV_{2 3} \del p \wedge \delbar p \otimes \del p \\
& + q^{(\alpha_s, \alpha_s)} \EV_{2 3} \opT_{1 2 3 4} \EV_{2 3} \delbar p \wedge \del p \otimes \del p \\
& - q^{(\alpha_s, \alpha_s)} q^{-(\omega_s, 2 \rho)} \del p \wedge \metMP.
\end{split}
\]
We can also write this as
\[
\begin{split}
\curv(\del p)
& = q^{(\alpha_s, \alpha_s)} \EV_{2 3} \opT_{1 2 3 4} \EV_{2 3} \del p \wedge \delbar p \otimes \del p \\
& - \EV_{2 3} \opT_{1 2 3 4} \EV_{2 3} \opT_{1 2 3 4}^{-1} \del p \wedge \delbar p \otimes \del p \\
& - (1 - q^{-(\alpha_s, \alpha_s)}) q^{-(\omega_s, 2 \rho)} \wedge(\metPM) \otimes \del p \\
& - q^{(\alpha_s, \alpha_s)} q^{-(\omega_s, 2 \rho)} \del p \wedge \metMP.
\end{split}
\]
\end{lemma}

\begin{proof}
According to \eqref{eq:connection}, the quantum Levi-Civita connection acts on $\del p$ by
\[
\conn(\del p) = q^{(\alpha_s, \alpha_s)} \EV_{2 3} \opT_{1 2 3 4} \delbar p \otimes \del p - q^{(\alpha_s, \alpha_s)} q^{-(\omega_s, 2 \rho)} p \metMP.
\]

\textit{First expression.}
Using the identity $(\diff \otimes \id)(\metMP) = 0$ from \eqref{eq:kahler-property} we compute
\[
\begin{split}
(\diff \otimes \id) \conn(\del p)
& = q^{(\alpha_s, \alpha_s)} \EV_{2 3} \opT_{1 2 3 4} \del \delbar p \otimes \del p - q^{(\alpha_s, \alpha_s)} q^{-(\omega_s, 2 \rho)} \diff p \wedge \metMP \\
& = q^{(\alpha_s, \alpha_s)} \EV_{2 3} \opT_{1 2 3 4} \EV_{2 3} \del p \wedge \delbar p \otimes \del p + q^{(\alpha_s, \alpha_s)} \EV_{2 3} \opT_{1 2 3 4} \EV_{2 3} \delbar p \wedge \del p \otimes \del p \\
& - q^{(\alpha_s, \alpha_s)} q^{-(\omega_s, 2 \rho)} \diff p \wedge \metMP.
\end{split}
\]
Here we also used the identity for $\del \delbar p$ as in \cref{lem:riemann-minus}.

On the other hand, using $(\id \otimes \conn) (\metMP) = 0$ from \eqref{eq:metric-compatibility} we have
\[
\begin{split}
(\id \otimes \conn) \conn(\del p)
& = q^{(\alpha_s, \alpha_s)} \EV_{2 3} \opT_{1 2 3 4} \delbar p \otimes \conn(\del p) \\
& = q^{2 (\alpha_s, \alpha_s)} \EV_{2 3} \opT_{1 2 3 4} \EV_{4 5} \opT_{3 4 5 6} \delbar p \otimes \delbar p \otimes \del p - q^{2 (\alpha_s, \alpha_s)} q^{-(\omega_s, 2 \rho)} \EV_{2 3} \opT_{1 2 3 4} \delbar p \otimes p \metMP.
\end{split}
\]
The relations \eqref{eq:calculus-relations} and \eqref{eq:right-module-relations} give $\opS_{1 2 3} \delbar p p = q^{(\omega_s, \omega_s)} \delbar p p$.
Then, using this identity together with \eqref{eq:S-evaluation}, we can rewrite the second term above as
\[
\begin{split}
\EV_{2 3} \opT_{1 2 3 4} \delbar p p
& = \EV_{2 3} \opSt_{2 3 4} \opS_{1 2 3} \delbar p p
= q^{(\omega_s, \omega_s)} \EV_{2 3} \opSt_{2 3 4} \delbar p p \\
& = q^{-(\omega_s, 2 \rho)} \EVp_{3 4} \delbar p p = \delbar p.
\end{split}
\]

Now we focus on the first term, which we write as
\[
A = \EV_{2 3} \opT_{1 2 3 4} \EV_{4 5} \opT_{3 4 5 6} \delbar p \otimes \delbar p \otimes \del p.
\]
Using $\opT_{1 2 3 4} = \opSt_{2 3 4} \opS_{1 2 3}$ and $\EV_{2 3} \opSt_{2 3 4} = q^{-(\omega_s, \omega_s + 2 \rho)} \EVp_{3 4}$ from \eqref{eq:S-evaluation} we rewrite it as
\[
\begin{split}
A & = q^{-2 (\omega_s, \omega_s + 2 \rho)} \EVp_{3 4} \opS_{1 2 3} \EVp_{5 6} \opS_{3 4 5} \delbar p \otimes \delbar p \otimes \del p \\
& = q^{-2 (\omega_s, \omega_s + 2 \rho)} \EVp_{3 4} \EVp_{5 6} \opS_{1 2 3} \opS_{3 4 5} \delbar p \otimes \delbar p \otimes \del p.
\end{split}
\]
We have $\opS_{1 2 3} \delbar p \otimes \delbar p = q^{(\omega_s, \omega_s)} \delbar p \otimes \delbar p$ from \eqref{eq:tensor-algebra-identities}.
Using this together with the "braid equation" $\opS_{1 2 3} \opS_{3 4 5} \opS_{1 2 3} = \opS_{3 4 5} \opS_{1 2 3} \opS_{3 4 5}$ from \cref{prop:S-properties} we get
\[
A = q^{-(\omega_s, \omega_s)} q^{-2 (\omega_s, \omega_s + 2 \rho)} \EVp_{3 4} \EVp_{3 4} \opS_{3 4 5} \opS_{1 2 3} \opS_{3 4 5} \delbar p \otimes \delbar p \otimes \del p.
\]
Using \eqref{eq:S-evaluation} again this can be rewritten as
\[
\begin{split}
A & = q^{-(\omega_s, \omega_s)} q^{-(\omega_s, \omega_s + 2 \rho)} \EVp_{3 4} \EV_{4 5} \opS_{1 2 3} \opS_{3 4 5} \delbar p \otimes \delbar p \otimes \del p \\
& = q^{-(\omega_s, \omega_s)} q^{-(\omega_s, \omega_s + 2 \rho)} \EVp_{3 4} \opS_{1 2 3} \EV_{4 5} \opS_{3 4 5} \delbar p \otimes \delbar p \otimes \del p.
\end{split}
\]
Next, using the quadratic relation \eqref{eq:S-quadratic-relations} for $\opS_{3 4 5}$ and \eqref{eq:S-evaluation} we have
\[
\begin{split}
\EV_{4 5} \opS_{3 4 5} & = q^{2 (\omega_s, \omega_s) - (\alpha_s, \alpha_s)} \EV_{4 5} \opS_{3 4 5}^{-1} + q^{(\omega_s, \omega_s)} (1 - q^{-(\alpha_s, \alpha_s)}) \EV_{4 5} \\
& = q^{2 (\omega_s, \omega_s) - (\alpha_s, \alpha_s)} q^{-(\omega_s, \omega_s + 2 \rho)} \EVp_{3 4} + q^{(\omega_s, \omega_s)} (1 - q^{-(\alpha_s, \alpha_s)}) \EV_{4 5}.
\end{split}
\]
The first term gives zero when acting on $\delbar p \otimes \delbar p \otimes \del p$. Hence we obtain
\[
\begin{split}
A & = q^{-(\omega_s, \omega_s)} q^{-(\omega_s, \omega_s + 2 \rho)} q^{(\omega_s, \omega_s)} (1 - q^{-(\alpha_s, \alpha_s)}) \EVp_{3 4} \opS_{1 2 3} \EV_{4 5} \delbar p \otimes \delbar p \otimes \del p \\
& = (1 - q^{-(\alpha_s, \alpha_s)}) q^{-(\omega_s, 2 \rho)} \EVp_{3 4} \EV_{4 5} \delbar p \otimes \delbar p \otimes \del p \\
& = (1 - q^{-(\alpha_s, \alpha_s)}) q^{-(\omega_s, 2 \rho)} \delbar p \otimes \metMP.
\end{split}
\]

Plugging this expression back into $(\id \otimes \conn) \conn(\del p)$ we get
\[
\begin{split}
(\id \otimes \conn) \conn(\del p)
& = q^{2 (\alpha_s, \alpha_s)} (1 - q^{-(\alpha_s, \alpha_s)}) q^{-(\omega_s, 2 \rho)} \delbar p \otimes \metMP - q^{2 (\alpha_s, \alpha_s)} q^{-(\omega_s, 2 \rho)} \delbar p \otimes \metMP \\
& = - q^{(\alpha_s, \alpha_s)} q^{-(\omega_s, 2 \rho)} \delbar p \otimes \metMP.
\end{split}
\]
Finally combining this with the expression for $(\diff \otimes \id) \conn(\del p)$ we obtain
\[
\begin{split}
\curv(\del p)
& = q^{(\alpha_s, \alpha_s)} \EV_{2 3} \opT_{1 2 3 4} \EV_{2 3} \del p \wedge \delbar p \otimes \del p + q^{(\alpha_s, \alpha_s)} \EV_{2 3} \opT_{1 2 3 4} \EV_{2 3} \delbar p \wedge \del p \otimes \del p \\
& - q^{(\alpha_s, \alpha_s)} q^{-(\omega_s, 2 \rho)} \del p \wedge \metMP.
\end{split}
\]
Note that the term of the form $\delbar p \wedge \delbar p \otimes \del p$ has canceled out.

\textit{Second expression.}
It is useful to derive an alternative expression for $\curv(\del p)$, obtained by rewriting the term $\delbar p \wedge \del p$ in the form $\del p \wedge \delbar p$.
Using \cref{lem:degree-two-relations} we have
\[
\delbar p \wedge \del p = - q^{-(\alpha_s, \alpha_s)} \opT_{1 2 3 4}^{-1} \del p \wedge \delbar p + q^{-(\alpha_s, \alpha_s)} q^{-(\omega_s, 2 \rho)} \EVp_{3 4} \opT_{3 4 5 6}^{-1} p \del p \wedge \delbar p.
\]
Moreover according to \cref{lem:evaluation-T} we have
\[
\EVp_{1 2} \opT_{1 2 3 4}^{-1} \del p \otimes \delbar p
= - (1 - q^{(\alpha_s, \alpha_s)}) q^{(\omega_s, 2 \rho)} \EV_{2 3} \del p \otimes \delbar p + (1 - q^{(\alpha_s, \alpha_s)}) \metPM p.
\]
Combining these two identities we obtain
\[
\begin{split}
\delbar p \wedge \del p & =
- q^{-(\alpha_s, \alpha_s)} \opT_{1 2 3 4}^{-1} \del p \wedge \delbar p + (1 - q^{-(\alpha_s, \alpha_s)}) \EV_{4 5} p \del p \wedge \delbar p \\
& - (1 - q^{-(\alpha_s, \alpha_s)}) q^{-(\omega_s, 2 \rho)} p p \wedge(\metPM).
\end{split}
\]
Applying $\EV_{2 3}$ and using \eqref{eq:evaluations-calculi} we get
\[
\begin{split}
\EV_{2 3} \delbar p \wedge \del p & =
- q^{-(\alpha_s, \alpha_s)} \EV_{2 3} \opT_{1 2 3 4}^{-1} \del p \wedge \delbar p - (1 - q^{-(\alpha_s, \alpha_s)}) q^{-(\omega_s, 2 \rho)} \wedge(\metPM) p.
\end{split}
\]

Now plugging this into the previous expression for $\curv(\del p)$ gives
\[
\begin{split}
\curv(\del p)
& = q^{(\alpha_s, \alpha_s)} \EV_{2 3} \opT_{1 2 3 4} \EV_{2 3} \del p \wedge \delbar p \otimes \del p \\
& - \EV_{2 3} \opT_{1 2 3 4} \EV_{2 3} \opT_{1 2 3 4}^{-1} \del p \wedge \delbar p \otimes \del p \\
& + (1 - q^{(\alpha_s, \alpha_s)}) q^{-(\omega_s, 2 \rho)} \EV_{2 3} \opT_{1 2 3 4} \wedge(\metPM) \otimes p \del p \\
& - q^{(\alpha_s, \alpha_s)} q^{-(\omega_s, 2 \rho)} \del p \wedge \metMP.
\end{split}
\]
Finally we can use \eqref{eq:right-module-relations} to compute
\[
\EV_{2 3} \opT_{1 2 3 4} p \del p = q^{-(\alpha_s, \alpha_s)} \EV_{2 3} \del p p = q^{-(\alpha_s, \alpha_s)} \del p,
\]
which allows us to rewrite the third term in the claimed form.
\end{proof}

\begin{remark}
It can be shown that the term $\EV_{2 3} \opT_{1 2 3 4} \EV_{2 3} \del p \wedge \delbar p \otimes \del p$ vanishes in the classical case.
On the other hand it is non-zero in the quantum case, as we shall see.
\end{remark}

It follows from our computations for the Riemann tensor that
\[
\curv \in \calc^{(1, 1)} \otimes \calc,
\]
where $\calc^{(1, 1)} \subset \calc^2$ denotes the subspace of elements of the form $\del p \wedge \delbar p$ (or equivalently of the form $\delbar p \wedge \del p$).
This is analogous to what we have classically for a \emph{Kähler manifold}.

\subsection{Bimodule map}

As mentioned before, the Riemann tensor $\curv$ is not a bimodule map in general.
However this turns out to be the case here, as we now show.

\begin{theorem}
\label{thm:R-bimodule}
We have that $\curv: \calc \to \calc^2 \otimes_\subalg \calc$ is a $\subalg$-bimodule map.
\end{theorem}

\begin{proof}
We only need to show that it is a right $\subalg$-bimodule map.
We are going to make extensive use of the identity $\EV_{2 3} \opT_{3 4 5 6} \opT_{1 2 3 4} = \opT_{1 2 3 4} \EV_{4 5}$, see for instance \cite[Lemma C.1]{paper-projective}.

\textit{Case $\delbar p$.}
Using the expression from \cref{lem:riemann-minus}, the right $\subalg$-module relations \eqref{eq:right-module-relations} and the fact that $\curv$ is a left $\subalg$-module map, we compute
\[
\begin{split}
\curv(\delbar p p) & = q^{-(\alpha_s, \alpha_s)} \opT_{1 2 3 4} \curv(p \delbar p) = q^{-(\alpha_s, \alpha_s)} \opT_{1 2 3 4} p \curv(\delbar p) \\
& = - q^{-(\alpha_s, \alpha_s)} \opT_{1 2 3 4} \EV_{4 5} \EV_{4 5} p \delbar p \wedge \del p \otimes \delbar p - q^{-(\omega_s, 2 \rho)} q^{-(\alpha_s, \alpha_s)} \opT_{1 2 3 4} p \delbar p \wedge \metPM.
\end{split}
\]
Consider $A_1 = \opT_{1 2 3 4} \EV_{4 5} \EV_{4 5} p \delbar p \wedge \del p \otimes \delbar p$. Using $\opT_{1 2 3 4} \EV_{4 5} = \EV_{2 3} \opT_{3 4 5 6} \opT_{1 2 3 4}$ we get
\[
\begin{split}
A_1 & = \EV_{2 3} \opT_{3 4 5 6} \opT_{1 2 3 4} \EV_{4 5} p \delbar p \wedge \del p \otimes \delbar p \\
& = \EV_{2 3} \opT_{3 4 5 6} \EV_{2 3} \opT_{3 4 5 6} \opT_{1 2 3 4} p \delbar p \wedge \del p \otimes \delbar p.
\end{split}
\]
Then using the right $\subalg$-module relations \eqref{eq:right-module-relations} we obtain
\[
\begin{split}
A_1 & = q^{(\alpha_s, \alpha_s)} \EV_{2 3} \opT_{3 4 5 6} \EV_{2 3} \opT_{5 6 7 8}^{-1} \delbar p \wedge \del p \otimes \delbar p p \\
& = q^{(\alpha_s, \alpha_s)} \EV_{2 3} \EV_{2 3} \delbar p \wedge \del p \otimes \delbar p p.
\end{split}
\]
On the other hand for $A_2 = \opT_{1 2 3 4} p \delbar p \wedge \metPM$ we compute
\[
A_2 = q^{(\alpha_s, \alpha_s)} \delbar p p \wedge \metPM
= q^{(\alpha_s, \alpha_s)} \delbar p \wedge \metPM p.
\]
Here we used that $\metPM$ is central. Then we obtain
\[
\curv(\delbar p p) = - \EV_{2 3} \EV_{2 3} \delbar p \wedge \del p \otimes \delbar p p - q^{-(\omega_s, 2 \rho)} \delbar p \wedge \metPM p = \curv(\delbar p) p.
\]

\textit{Case $\del p$.}
Similarly, using the expression from \cref{lem:riemann-plus} we compute
\[
\begin{split}
\curv(\del p p) & = q^{(\alpha_s, \alpha_s)} \opT_{1 2 3 4} \curv(p \del p) = q^{(\alpha_s, \alpha_s)} \opT_{1 2 3 4} p \curv(\del p) \\
& = q^{2 (\alpha_s, \alpha_s)} \opT_{1 2 3 4} \EV_{4 5} \opT_{3 4 5 6} \EV_{4 5} p \del p \wedge \delbar p \otimes \del p \\
& + q^{2 (\alpha_s, \alpha_s)} \opT_{1 2 3 4} \EV_{4 5} \opT_{3 4 5 6} \EV_{4 5} p \delbar p \wedge \del p \otimes \del p \\
& - q^{2 (\alpha_s, \alpha_s)} q^{-(\omega_s, 2 \rho)} \opT_{1 2 3 4} p \del p \wedge \metMP.
\end{split}
\]
Let us write this expression in the form
\[
\curv(\del p p) = q^{2 (\alpha_s, \alpha_s)} B_1 + q^{2 (\alpha_s, \alpha_s)} B_2 - q^{2 (\alpha_s, \alpha_s)} q^{-(\omega_s, 2 \rho)} B_3.
\]
For the term $B_1$ we compute
\[
\begin{split}
B_1 & = \opT_{1 2 3 4} \EV_{4 5} \opT_{3 4 5 6} \EV_{4 5} p \del p \wedge \delbar p \otimes \del p \\
& = \EV_{2 3} \opT_{3 4 5 6} \opT_{1 2 3 4} \opT_{3 4 5 6} \EV_{4 5} p \del p \wedge \delbar p \otimes \del p.
\end{split}
\]
To proceed we use the "braid equation" $\opT_{3 4 5 6} \opT_{1 2 3 4} \opT_{3 4 5 6} = \opT_{1 2 3 4} \opT_{3 4 5 6} \opT_{1 2 3 4}$, which can be easily derived from the properties listed in \cref{prop:S-properties}. We get
\[
\begin{split}
B_1 & = \EV_{2 3} \opT_{1 2 3 4} \opT_{3 4 5 6} \opT_{1 2 3 4} \EV_{4 5} p \del p \wedge \delbar p \otimes \del p \\
& = \EV_{2 3} \opT_{1 2 3 4} \opT_{3 4 5 6} \EV_{2 3} \opT_{3 4 5 6} \opT_{1 2 3 4} p \del p \wedge \delbar p \otimes \del p \\
& = q^{-(\alpha_s, \alpha_s)} \EV_{2 3} \opT_{1 2 3 4} \opT_{3 4 5 6} \EV_{2 3} \opT_{5 6 7 8}^{-1} \del p \wedge \delbar p \otimes \del p p \\
& = q^{-(\alpha_s, \alpha_s)} \EV_{2 3} \opT_{1 2 3 4} \EV_{2 3} \del p \wedge \delbar p \otimes \del p p.
\end{split}
\]
The computation for $B_2$ is essentially identical and leads to
\[
B_2 = q^{-(\alpha_s, \alpha_s)} \EV_{2 3} \opT_{1 2 3 4} \EV_{2 3} \delbar p \wedge \del p \otimes \del p p.
\]
Finally we easily get $B_3 = \del p \wedge \metMP p$. This leads to $\curv(\del p p) = \curv(\del p) p$.
\end{proof}

\subsection{Some consequences}

In this subsection we explore various consequences of the fact that $\curv$ is a $\subalg$-bimodule map.
First we consider the possibility of extending the generalized braiding $\sigma$ to a map $\calc \otimes_\subalg \Omega^n \to \Omega^n \otimes_\subalg \calc$, following \cite[Definition 4.10]{quantum-book}.

\begin{definition}
An $A$-bimodule $E$ with left bimodule connection $(\nabla_E, \sigma_E)$ is called \emph{extendable} if the generalized braiding $\sigma_E: E \otimes_A \Omega^1 \to \Omega^1 \otimes_A E$ extends to a map
\[
\sigma_E^{[n]}: E \otimes_A \Omega^n \to \Omega^n \otimes_A E
\]
for all $n \geq 1$, in such a way that on $E \otimes_A \Omega^m \otimes_A \Omega^n$ we have
\[
\sigma_E^{[m + n]} \circ (\id \otimes \wedge) = (\wedge \otimes \id) \circ (\id \otimes \sigma_E^{[n]}) \circ (\sigma_E^{[m]} \otimes \id).
\]
\end{definition}

It is easy to see that this extension is unique, provided it exists.

This notion of extendability features in the definition of a certain (DG) category ${}_A \mathcal{G}_A$, defined by Beggs and Majid in \cite[Section 4.1]{quantum-book}.
The objects of this category are triples $(E, \nabla_E, \sigma_E)$, where $E$ is an $A$-bimodule and $(\nabla_E, \sigma_E)$ is a left bimodule connection such that: 1) $\sigma_E$ is extendable; 2) the curvature $R_E$ is a bimodule map.

\begin{proposition}
\label{prop:sigma-extendable}
The generalized braiding $\sigma: \calc \otimes_\subalg \calc \to \calc \otimes_\subalg \calc$ is extendable.
\end{proposition}

\begin{proof}
We make use of the following result (\cite[Lemma 4.14]{quantum-book}, taking into account that the universal differential calculus is called maximal prolongation there): provided we use the universal differential calculus of a given FODC, if $(E, \nabla_E, \sigma_E)$ is a left bimodule connection with $R_E$ being a bimodule map, then $\sigma_E$ is automatically extendable.

In our case we are considering the Heckenberger-Kolb calculus $\calc^\bullet$, which is the universal differential calculus of the FODC $\calc$.
Since $(\conn, \sigma)$ is a left bimodule connection and $\curv$ is a bimodule map by \cref{thm:R-bimodule}, we obtain the result.
\end{proof}

Therefore $(\calc, \conn, \sigma)$ is an object in the category ${}_\subalg \mathcal{G}_\subalg$ mentioned above.
One consequence of this fact is the following identity, which in the classical case amounts to the \emph{antisymmetry} property $R_{i j k l} = - R_{j i k l}$ for the components of the Riemann tensor.

\begin{corollary}
We have the identity
\[
(\curv \otimes \id + (\sigma^{[2]} \otimes \id) \circ (\id \otimes \curv)) (\met) = 0.
\]
\end{corollary}

\begin{proof}
This follows from \cite[Corollary 4.16]{quantum-book}, which gives the stated result provided that $(\calc, \conn, \sigma)$ is an object in the category ${}_\subalg \mathcal{G}_\subalg$ and that the connection $\conn$ is compatible with the quantum metric $\met$, that is $(\conn \otimes \id + (\sigma \otimes \id) \circ (\id \otimes \conn)) (g) = 0$.
Compatibility of $\conn$ with $g$ for the quantum projective spaces was proven in \cite[Theorem 8.4]{paper-projective}.
\end{proof}

Finally we come back to the generalized braiding $\sigma$.
We show that it satisfies the braid equation possibly up to symmetric terms, in the following sense.

\begin{proposition}
The generalized braiding $\sigma: \calc \otimes_\subalg \calc \to \calc \otimes_\subalg \calc$ satisfies
\[
(\wedge \otimes \id) \circ \sigma_1 \circ \sigma_2 \circ \sigma_1 = (\wedge \otimes \id) \circ \sigma_2 \circ \sigma_1 \circ \sigma_2
\]
where we use the standard leg-notation $\sigma_1 = \sigma \otimes \id$ and $\sigma_2 = \id \otimes \sigma$. 
\end{proposition}

\begin{proof}
Since $\sigma$ is extendable by \cref{prop:sigma-extendable}, we have the identity
\[
\sigma^{[2]} \circ (\id \otimes \wedge) = (\wedge \otimes \id) \circ \sigma_2 \circ \sigma_1.
\]
We have $\wedge \circ (\id + \sigma) = 0$ by \cref{cor:compatibility-wedge}. This leads to
\[
- (\wedge \otimes \id) \circ \sigma_2 \circ \sigma_1 = (\wedge \otimes \id) \circ \sigma_2 \circ \sigma_1 \circ \sigma_2.
\]
Finally using $- \wedge = \wedge \circ \sigma$ on the left-hand side gives the result.
\end{proof}

\begin{remark}
We leave open the question of whether $\sigma$ actually satisfies the braid equation.
It is easy to show that $\sigma_1 \circ \sigma_2 \circ \sigma_1 = \sigma_2 \circ \sigma_1 \circ \sigma_2$ holds when acting on $\calcP \otimes \calcP \otimes \calcP$ or $\calcM \otimes \calcM \otimes \calcM$, hence the difficult part is to show that it holds for the mixed terms.
\end{remark}

\section{Ricci tensor}
\label{sec:ricci-tensor}

In this section we study the Ricci tensor, corresponding to the quantum metric $\met$ and the quantum Levi-Civita connection $\nabla$.
Its definition depends on the auxiliary choice of a splitting map $\splitmap$ for the wedge product, for which we use the one-parameter family introduced in \cref{prop:split-map}.
We are going to show that this free parameter can be fixed by requiring the Ricci tensor to be symmetric, that is $\wedge(\ricci) = 0$.
Moreover, with this choice the Ricci tensor turns out to be proportional to the quantum metric $\met$.
This gives a quantum analogue of the \emph{Einstein condition} for the quantum Fubini-Study metric.

\subsection{Definition and computation}

We define the Ricci tensor as an appropriate "contraction" of the Riemann tensor, following \cite[Section 8.1]{quantum-book}.

\begin{definition}
The \emph{Ricci tensor} is the element $\ricci \in \calc \otimes_\subalg \calc$ defined by
\[
\ricci := ((\cdot, \cdot) \otimes \id \otimes \id) \circ (\id \otimes \splitmap \otimes \id) \circ (\id \otimes \curv) (\met).
\]
\end{definition}

Here $\imet$ is the inverse of the quantum metric $\met$, while $\splitmap$ is a splitting map for the wedge product.
In the classical case, with $\splitmap$ being the antisymmetrization map, it corresponds to the usual Ricci tensor up to an overall factor (see \cite[Example 8.10]{quantum-book}).

In our case, we are going to use the splitting map $\splitmap$ introduced in \cref{prop:split-map}.
This is actually a one-parameter family of splitting maps, since we have $c_{+ -} + c_{- +} = 1$ and
\[
\begin{split}
\splitmap(\del p \wedge \delbar p) & = (c_{+ -} \id - c_{- +} \sigmaPM) (\del p \otimes \delbar p), \\
\splitmap(\delbar p \wedge \del p) & = (c_{- +} \id - c_{+ -} \sigmaMP) (\delbar p \otimes \del p).
\end{split}
\]

Corresponding to the decomposition $\met = \metPM + \metMP$, we also write
\[
\ricci = \ricciPM + \ricciMP.
\]
In the following we are going to explicitly compute these two components of the Ricci tensor.
We begin with $\ricciPM$, which is technically easier.

\begin{lemma}
\label{lem:ricciPM}
We have
\[
\ricciPM = - \cMP q^{-2 (\omega_s, 2 \rho)} \qdim(V) \metPM.
\]
\end{lemma}

\begin{proof}
The component $\ricciPM$ is defined by
\[
\ricciPM = ((\cdot, \cdot) \otimes \id \otimes \id) \circ (\id \otimes \splitmap \otimes \id) \circ (\id \otimes \curv) (\metPM).
\]
Using the expression for $\curv(\delbar p)$ from \cref{lem:riemann-minus} we compute
\[
\begin{split}
(\id \otimes \curv) (\metPM)
& = \EVp_{1 2} \EV_{2 3} (\id \otimes R_\nabla) (\del p \otimes \delbar p) \\
& = - \EVp_{1 2} \EV_{2 3} \EV_{4 5} \EV_{4 5} \del p \otimes \delbar p \wedge \del p \otimes \delbar p \\
& - q^{-(\omega_s, 2 \rho)} \EVp_{1 2} \EV_{2 3} \EVp_{5 6} \EV_{6 7} \del p \otimes \delbar p \wedge \del p \otimes \delbar p.
\end{split}
\]
Now consider the splitting map $\splitmap(\delbar p \wedge \del p) = (\cMP \id - \cPM \sigmaMP)(\delbar p \otimes \del p)$ from \cref{prop:split-map}.
Taking into account that $(\del p, \del p) = 0$ we obtain
\[
((\cdot, \cdot) \otimes \id) \circ (\id \otimes \splitmap) (\del p \otimes \delbar p \wedge \del p)
= \cMP (\del p, \delbar p) \del p.
\]
Using this result in the computation of $\ricciPM$ we obtain
\[
\begin{split}
\ricciPM
& = - \cMP \EVp_{1 2} \EV_{2 3} \EV_{4 5} \EV_{4 5} (\del p, \delbar p) \del p \otimes \delbar p \\
& - \cMP q^{-(\omega_s, 2 \rho)} \EVp_{1 2} \EV_{2 3} \EVp_{5 6} \EV_{6 7} (\del p, \delbar p) \del p \otimes \delbar p.
\end{split}
\]
Let us write this in the form $\ricciPM = - \cMP (A_1 + A_2)$.
We are now going to obtain simpler expressions for the terms $A_1$ and $A_2$.

\textit{The term $A_1$.}
Using the expression for inverse metric from \eqref{eq:inverse-metric} we write
\[
\begin{split}
A_1 & = \EVp_{1 2} \EV_{2 3} \EV_{4 5} \EV_{4 5} (\del p, \delbar p) \del p \otimes \delbar p \\
& = q^{(\alpha_s, \alpha_s)} q^{-(\omega_s, \omega_s + 2 \rho)} \EVp_{1 2} \EV_{2 3} \EV_{4 5} \EV_{4 5} \opS_{1 2 3} \CV_3 p \del p \otimes \delbar p \\
& - q^{(\alpha_s, \alpha_s)} q^{-(\omega_s, 2 \rho)} \EVp_{1 2} \EV_{2 3} \EV_{4 5} \EV_{4 5} p p \del p \otimes \delbar p.
\end{split}
\]
The second term vanishes, since using the algebraic relations we compute
\[
\EV_{2 3} \EV_{4 5} \EV_{4 5} p p \del p \otimes \delbar p = \EV_{2 3} \EV_{2 3} \EV_{2 3} p p \del p \otimes \delbar p
= \EV_{2 3} \EV_{2 3} p \del p \otimes \delbar p = 0.
\]
For the first term, using the duality relation \eqref{eq:duality} we get
\[
\EV_{4 5} \opS_{1 2 3} \CV_3 = \opS_{1 2 3} \EV_{4 5} \CV_3 = \opS_{1 2 3}.
\]
Using this and \eqref{eq:right-module-relations} we obtain
\[
\EV_{4 5} \opS_{1 2 3} \CV_3 p \del p \otimes \delbar p
= \opS_{1 2 3} p \del p \otimes \delbar p
= q^{(\omega_s, \omega_s) - (\alpha_s, \alpha_s)} \del p p \otimes \delbar p.
\]
Finally we obtain
\[
\begin{split}
& \EVp_{1 2} \EV_{2 3} \EV_{4 5} \EV_{4 5} \opS_{1 2 3} \CV_3 p \del p \otimes \delbar p \\
& = q^{(\omega_s, \omega_s) - (\alpha_s, \alpha_s)} \EVp_{1 2} \EV_{2 3} \EV_{4 5} \del p p \otimes \delbar p
= q^{(\omega_s, \omega_s) - (\alpha_s, \alpha_s)} \EVp_{1 2} \EV_{2 3} \EV_{2 3} \del p p \otimes \delbar p \\
& = q^{(\omega_s, \omega_s) - (\alpha_s, \alpha_s)} \EVp_{1 2} \EV_{2 3} \del p \otimes \delbar p
= q^{(\omega_s, \omega_s) - (\alpha_s, \alpha_s)} \metPM.
\end{split}
\]
From this we conclude that $A_1 = q^{-(\omega_s, 2 \rho)} \metPM$.

\textit{The term $A_2$.} Now consider the term
\[
A_2 = q^{-(\omega_s, 2 \rho)} \EVp_{1 2} \EV_{2 3} \EVp_{5 6} \EV_{6 7} (\del p, \delbar p) \del p \otimes \delbar p
= q^{-(\omega_s, 2 \rho)} \EVp_{1 2} \EV_{2 3} (\del p, \delbar p) \metPM.
\]
We have $\EVp_{1 2} \EV_{2 3} (\del p, \delbar p) = \Trq(\metPM)$. According to \cref{prop:quantum-trace} we have
\[
\Trq(\metPM) = q^{-(\omega_s, 2 \rho)} \qdim(V) - 1.
\]
Therefore we obtain
\[
A_2 = q^{-(\omega_s, 2 \rho)} (q^{-(\omega_s, 2 \rho)} \qdim(V) - 1) \metPM.
\]

\textit{The sum.}
Finally, since $\ricciPM = - \cMP (A_1 + A_2)$ we obtain
\[
\begin{split}
\ricciPM & = - \cMP \left( q^{-(\omega_s, 2 \rho)} + q^{-(\omega_s, 2 \rho)} (q^{-(\omega_s, 2 \rho)} \qdim(V) - 1) \right) \metPM \\
& = - \cMP q^{-2 (\omega_s, 2 \rho)} \qdim(V) \metPM. \qedhere
\end{split}
\]
\end{proof}

Next we compute $\ricciMP$, which is technically more involved.

\begin{lemma}
\label{lem:ricciMP}
We have
\[
\ricciMP = - \cPM q^{(\alpha_s, \alpha_s)} \qdim(V) \metMP.
\]
\end{lemma}

\begin{proof}
Our goal is to compute the expression
\[
\ricciMP = \EVp_{1 2} \EV_{2 3} ((\cdot, \cdot) \otimes \id \otimes \id) \circ (\id \otimes \splitmap \otimes \id) (\delbar p \otimes \curv(\del p)).
\]
For $\curv(\del p)$ we use the second expression derived in \cref{lem:riemann-plus}, namely
\[
\begin{split}
\curv(\del p) & = q^{(\alpha_s, \alpha_s)} \EV_{2 3} \opT_{1 2 3 4} \EV_{2 3} \del p \wedge \delbar p \otimes \del p - \EV_{2 3} \opT_{1 2 3 4} \EV_{2 3} \opT_{1 2 3 4}^{-1} \del p \wedge \delbar p \otimes \del p \\
& - (1 - q^{-(\alpha_s, \alpha_s)}) q^{-(\omega_s, 2 \rho)} \wedge(\metPM) \otimes \del p - q^{(\alpha_s, \alpha_s)} q^{-(\omega_s, 2 \rho)} \del p \wedge \metMP.
\end{split}
\]
Now consider the splitting map $\splitmap(\del p \wedge \delbar p) = (\cPM \id - \cMP \sigmaPM)(\del p \otimes \delbar p)$ from \cref{prop:split-map}.
Taking into account that $(\delbar p, \delbar p) = 0$ we obtain
\[
((\cdot, \cdot) \otimes \id) \circ (\id \otimes \splitmap) (\delbar p \otimes \del p \wedge \delbar p)
= \cPM (\delbar p, \del p) \delbar p.
\]
Using this in our expression for $\ricciMP$ we get
\[
\begin{split}
\ricciMP & = \cPM q^{(\alpha_s, \alpha_s)} \EVp_{1 2} \EV_{2 3} \EV_{4 5} \opT_{3 4 5 6} \EV_{4 5} (\delbar p, \del p) \delbar p \otimes \del p \\
& - \cPM \EVp_{1 2} \EV_{2 3} \EV_{4 5} \opT_{3 4 5 6} \EV_{4 5} \opT_{3 4 5 6}^{-1} (\delbar p, \del p) \delbar p \otimes \del p \\
& - \cPM (1 - q^{-(\alpha_s, \alpha_s)}) q^{-(\omega_s, 2 \rho)} \EVp_{1 2} \EV_{2 3} \EVp_{3 4} \EV_{4 5} (\delbar p, \del p) \delbar p \otimes \del p \\
& - \cPM q^{(\alpha_s, \alpha_s)} q^{-(\omega_s, 2 \rho)} \EVp_{1 2} \EV_{2 3} (\delbar p, \del p) \metMP.
\end{split}
\]
With obvious notation, we can write the expression above as
\[
\ricciMP = - \cPM (A_1 + A_2 + A_3 + A_4).
\]

\textit{The term $A_1$.}
We begin with the term
\[
A_1 = - q^{(\alpha_s, \alpha_s)} \EVp_{1 2} \EV_{2 3} \EV_{4 5} \opT_{3 4 5 6} \EV_{4 5} (\delbar p, \del p) \delbar p \otimes \del p.
\]
Using the inverse metric from \eqref{eq:inverse-metric} and the relation \eqref{eq:evaluations-calculi} we compute
\[
\begin{split}
\EV_{4 5} (\delbar p, \del p) \delbar p
& = \EV_{4 5} \CVp_2 p \delbar p - q^{-(\omega_s, 2 \rho)} \EV_{4 5} p p \delbar p = \CVp_2 \EV_{2 3} p \delbar p - q^{-(\omega_s, 2 \rho)} p \delbar p \\
& = \CVp_2 \delbar p - q^{-(\omega_s, 2 \rho)} p \delbar p.
\end{split}
\]
This lets us rewrite $A_1$ in the form
\[
\begin{split}
A_1 & = - q^{(\alpha_s, \alpha_s)} \EVp_{1 2} \EV_{2 3} \EV_{4 5} \opT_{3 4 5 6} \CVp_2 \delbar p \otimes \del p \\
& + q^{(\alpha_s, \alpha_s)} q^{-(\omega_s, 2 \rho)} \EVp_{1 2} \EV_{2 3} \EV_{4 5} \opT_{3 4 5 6} p \delbar p \otimes \del p.
\end{split}
\]

Write $A_{1, 1} = \EVp_{1 2} \EV_{2 3} \EV_{4 5} \opT_{3 4 5 6} \CVp_2 \delbar p \otimes \del p$. Using the definition of $\opT_{3 4 5 6}$ we have
\[
A_{1, 1} = \EVp_{1 2} \EV_{2 3} \EV_{4 5} \opS_{3 4 5} \opSt_{4 5 6} \CVp_2 \delbar p \otimes \del p
= \EVp_{1 2} \EV_{2 3} \EV_{4 5} \opS_{3 4 5} \CVp_2 \opSt_{2 3 4} \delbar p \otimes \del p.
\]
Next, using the quadratic relation \eqref{eq:S-quadratic-relations} for $\opS_{3 4 5}$ and \eqref{eq:S-evaluation}, we compute
\[
\begin{split}
A_{1, 1} & = q^{2 (\omega_s, \omega_s) - (\alpha_s, \alpha_s)} \EVp_{1 2} \EV_{2 3} \EV_{4 5} \opS_{3 4 5}^{-1} \CVp_2 \opSt_{2 3 4} \delbar p \otimes \del p \\
& + (1 - q^{-(\alpha_s, \alpha_s)}) q^{(\omega_s, \omega_s)} \EVp_{1 2} \EV_{2 3} \EV_{4 5} \CVp_2 \opSt_{2 3 4} \delbar p \otimes \del p \\
& = q^{-(\omega_s, 2 \rho)} q^{(\omega_s, \omega_s) - (\alpha_s, \alpha_s)} \EVp_{1 2} \EV_{2 3} \EVp_{3 4} \CVp_2 \opSt_{2 3 4} \delbar p \otimes \del p \\
& + (1 - q^{-(\alpha_s, \alpha_s)}) q^{(\omega_s, \omega_s)} \EVp_{1 2} \EV_{2 3} \CVp_2 \EV_{2 3} \opSt_{2 3 4} \delbar p \otimes \del p.
\end{split}
\]
The second term vanishes after using \eqref{eq:S-evaluation}. For the first term we use the duality relation $\EVp_{3 4} \CVp_2 = \id$ from \eqref{eq:duality}.
Then making use of \eqref{eq:S-evaluation} again we obtain
\[
A_{1, 1} = q^{-(\omega_s, 2 \rho)} q^{(\omega_s, \omega_s) - (\alpha_s, \alpha_s)} \EVp_{1 2} \EV_{2 3} \opSt_{2 3 4} \delbar p \otimes \del p = 0.
\]

Now consider the term $A_{1, 2} = \EVp_{1 2} \EV_{2 3} \EV_{4 5} \opT_{3 4 5 6} p \delbar p \otimes \del p$.
Using the right $\subalg$-module relations \eqref{eq:right-module-relations} and the commutation relations between $\opS$ and $\opSt$ we obtain
\[
\begin{split}
A_{1, 2} & = q^{(\alpha_s, \alpha_s) - (\omega_s, \omega_s)} \EVp_{1 2} \EV_{2 3} \EV_{4 5} \opS_{3 4 5} \opSt_{4 5 6} \opSt_{2 3 4}^{-1} \delbar p p \otimes \del p \\
& = q^{(\alpha_s, \alpha_s) - (\omega_s, \omega_s)} \EVp_{1 2} \EV_{2 3} \EV_{4 5} \opSt_{4 5 6} \opSt_{2 3 4}^{-1} \opS_{3 4 5} \delbar p \otimes p \del p \\
& = \EVp_{1 2} \EV_{2 3} \EV_{4 5} \opSt_{4 5 6} \opSt_{2 3 4}^{-1} \delbar p \otimes \del p p.
\end{split}
\]
Next, using \eqref{eq:S-evaluation} and \eqref{eq:algebra-relations} leads to
\[
A_{1, 2} = q^{-(\omega_s, \omega_s + 2 \rho)} \EVp_{1 2} \EV_{2 3} \EVp_{5 6} \opSt_{2 3 4}^{-1} \delbar p \otimes \del p p
= q^{-(\omega_s, \omega_s)} \EVp_{1 2} \EV_{2 3} \opSt_{2 3 4}^{-1} \delbar p \otimes \del p.
\]
Finally, using \eqref{eq:S-quadratic-relations} for $\opSt_{2 3 4}^{-1}$, we get
\[
A_{1, 2} = (1 - q^{-(\alpha_s, \alpha_s)}) \EVp_{1 2} \EV_{2 3} \delbar p \otimes \del p = (1 - q^{-(\alpha_s, \alpha_s)}) \metMP.
\]

Since $A_1 = - q^{(\alpha_s, \alpha_s)} A_{1, 1} + q^{(\alpha_s, \alpha_s)} q^{-(\omega_s, 2 \rho)} A_{1, 2}$, we obtain
\[
A_1 = - (1 - q^{(\alpha_s, \alpha_s)}) q^{-(\omega_s, 2 \rho)} \metMP.
\]

\textit{The term $A_2$.}
Now consider the term
\[
A_2 = \EVp_{1 2} \EV_{2 3} \EV_{4 5} \opT_{3 4 5 6} \EV_{4 5} \opT_{3 4 5 6}^{-1} (\delbar p, \del p) \delbar p \otimes \del p.
\]

We claim that $\EV_{4 5} \opT_{3 4 5 6}^{-1} (\delbar p, \del p) \delbar p = q^{-(\omega_s, 2 \rho)} \delbar p p$.
First, by \eqref{eq:S-evaluation} we have
\[
\EV_{4 5} \opT_{3 4 5 6}^{-1} = \EV_{4 5} \opS_{3 4 5}^{-1} \opSt_{4 5 6}^{-1} = q^{-(\omega_s, \omega_s + 2 \rho)} \EVp_{3 4} \opSt_{4 5 6}^{-1}.
\]
This leads to the expression
\[
\EV_{4 5} \opT_{3 4 5 6}^{-1} (\delbar p, \del p) \delbar p = q^{-(\omega_s, \omega_s + 2 \rho)} \EVp_{3 4} \opSt_{4 5 6}^{-1} \CVp_2 p \delbar p - q^{-(\omega_s, 2 \rho)} q^{-(\omega_s, \omega_s + 2 \rho)} \EVp_{3 4} \opSt_{4 5 6}^{-1} p p \delbar p.
\]
For the first term we use the duality relation \eqref{eq:duality} and get
\[
\EVp_{3 4} \opSt_{4 5 6}^{-1} \CVp_2 p \delbar p = \EVp_{3 4} \CVp_2 \opSt_{2 3 4}^{-1} p \delbar p = \opSt_{2 3 4}^{-1} p \delbar p.
\]
For the second term we use the quadratic relation \eqref{eq:S-quadratic-relations} for $\opSt_{4 5 6}^{-1}$ and \eqref{eq:right-module-relations}. We get
\[
\begin{split}
\EVp_{3 4} \opSt_{4 5 6}^{-1} p p \delbar p
& = q^{2 (\omega_s, \omega_s) - (\alpha_s, \alpha_s)} \EVp_{3 4} \opSt_{4 5 6} p p \delbar p + (1 - q^{-(\alpha_s, \alpha_s)}) q^{(\omega_s, \omega_s)} \EVp_{3 4} p p \delbar p \\
& = q^{(\omega_s, \omega_s)} \EVp_{3 4} p \delbar p p + (1 - q^{-(\alpha_s, \alpha_s)}) q^{(\omega_s, \omega_s + 2 \rho)} p \delbar p \\
& = (1 - q^{-(\alpha_s, \alpha_s)}) q^{(\omega_s, \omega_s + 2 \rho)} p \delbar p.
\end{split}
\]
Putting these identities together we obtain
\[
\EV_{4 5} \opT_{3 4 5 6}^{-1} (\delbar p, \del p) \delbar p = q^{-(\omega_s, \omega_s + 2 \rho)} \opSt_{2 3 4}^{-1} p \delbar p - (1 - q^{-(\alpha_s, \alpha_s)}) q^{-(\omega_s, 2 \rho)} p \delbar p.
\]
Then, using the quadratic relation \eqref{eq:S-quadratic-relations} once more for $\opSt_{2 3 4}^{-1}$, we get
\[
\EV_{4 5} \opT_{3 4 5 6}^{-1} (\delbar p, \del p) \delbar p = q^{(\omega_s, \omega_s) - (\alpha_s, \alpha_s)} q^{-(\omega_s, 2 \rho)} \opSt_{2 3 4} p \delbar p = q^{-(\omega_s, 2 \rho)} \delbar p p.
\]

Finally, using this together with \eqref{eq:right-module-relations} and \eqref{eq:evaluations-calculi}, we compute
\[
\begin{split}
A_2 & = \EVp_{1 2} \EV_{2 3} \EV_{4 5} \opT_{3 4 5 6} \EV_{4 5} \opT_{3 4 5 6}^{-1} (\delbar p, \del p) \delbar p \otimes \del p \\
& = q^{-(\omega_s, 2 \rho)} \EVp_{1 2} \EV_{2 3} \EV_{4 5} \opT_{3 4 5 6} \delbar p p \otimes \del p \\
& = q^{-(\alpha_s, \alpha_s)} q^{-(\omega_s, 2 \rho)} \EVp_{1 2} \EV_{2 3} \EV_{4 5} \delbar p \otimes \del p p \\
& = q^{-(\alpha_s, \alpha_s)} q^{-(\omega_s, 2 \rho)} \EVp_{1 2} \EV_{2 3} \delbar p \otimes \del p \\
& = q^{-(\alpha_s, \alpha_s)} q^{-(\omega_s, 2 \rho)} \metMP.
\end{split}
\]

\textit{The term $A_3$.}
Now we consider the term
\[
A_3 = (1 - q^{-(\alpha_s, \alpha_s)}) q^{-(\omega_s, 2 \rho)} \EVp_{1 2} \EV_{2 3} \EVp_{3 4} \EV_{4 5} (\delbar p, \del p) \delbar p \otimes \del p.
\]
Using \eqref{eq:algebra-relations}, \eqref{eq:evaluations-calculi} and \eqref{eq:duality} we compute
\[
\begin{split}
\EVp_{3 4} \EV_{4 5} (\delbar p, \del p) \delbar p 
& = \EVp_{3 4} \EV_{4 5} \CVp_2 p \delbar p - q^{-(\omega_s, 2 \rho)} \EVp_{3 4} \EV_{4 5} p p \delbar p \\
& = \EVp_{3 4} \CVp_2 \EV_{2 3} p \delbar p - q^{-(\omega_s, 2 \rho)} \EVp_{3 4} p \delbar p = \delbar p.
\end{split}
\]
Therefore we obtain
\[
A_3 = (1 - q^{-(\alpha_s, \alpha_s)}) q^{-(\omega_s, 2 \rho)} \metMP.
\]

\textit{The term $A_4$.}
Finally consider the term
\[
A_4 = q^{(\alpha_s, \alpha_s)} q^{-(\omega_s, 2 \rho)} \EVp_{1 2} \EV_{2 3} (\delbar p, \del p) \metMP
= q^{(\alpha_s, \alpha_s)} q^{-(\omega_s, 2 \rho)} \Trq(\metMP) \metMP.
\]
We have $\Trq(\metMP) = q^{(\omega_s, 2 \rho)} \mathrm{qdim}(V) - 1$ from \cref{prop:quantum-trace}. Then we obtain
\[
A_4 = (q^{(\alpha_s, \alpha_s)} \mathrm{qdim}(V) - q^{(\alpha_s, \alpha_s)} q^{-(\omega_s, 2 \rho)}) \metMP.
\]

\textit{The sum.}
For the sum $A = A_1 + A_2 + A_3 + A_4$ we obtain
\[
\begin{split}
A & = - (1 - q^{(\alpha_s, \alpha_s)}) q^{-(\omega_s, 2 \rho)} \metMP + q^{-(\alpha_s, \alpha_s)} q^{-(\omega_s, 2 \rho)} \metMP \\
& + (1 - q^{-(\alpha_s, \alpha_s)}) q^{-(\omega_s, 2 \rho)} \metMP + (q^{(\alpha_s, \alpha_s)} \qdim(V) - q^{(\alpha_s, \alpha_s)} q^{-(\omega_s, 2 \rho)}) \metMP \\
& = q^{(\alpha_s, \alpha_s)} \qdim(V) \metMP.
\end{split}
\]
Since $\ricciMP = - c_{+ -} A$ we obtain the result.
\end{proof}

\subsection{Einstein condition and scalar curvature}

Since the Ricci tensor is symmetric in the classical case, we can ask for the condition $\wedge(\ricci) = 0$ in the quantum case.
This uniquely fixes the coefficients $\cPM$ and $\cMP$ appearing in the splitting map $\splitmap$, as in the next result.

\begin{lemma}
We have $\wedge(\ricci) = 0$ if and only if
\[
\cPM = \frac{1}{1 + q^{(\alpha_s, \alpha_s)} q^{2 (\omega_s, 2 \rho)}}, \quad
\cMP = \frac{1}{1 + q^{-(\alpha_s, \alpha_s)} q^{-2 (\omega_s, 2 \rho)}}.
\]
\end{lemma}

\begin{proof}
Using \cref{lem:ricciPM} and \cref{lem:ricciMP} we can write
\[
\ricci = - \cMP q^{- 2 (\omega_s, 2 \rho)} \qdim(V) \metPM - \cPM q^{(\alpha_s, \alpha_s)} \qdim(V) \metMP.
\]
We have $\wedge(\metMP) = - \wedge(\metPM)$, due to the symmetry property $\wedge(\met) = 0$ of the quantum metric from \eqref{eq:metric-symmetric}.
Then applying $\wedge$ to $\ricci$ gives
\[
\wedge(\ricci) = \qdim(V) (- \cMP q^{- 2 (\omega_s, 2 \rho)} + \cPM q^{(\alpha_s, \alpha_s)}) \wedge(\metPM).
\]
Since $\wedge(\metPM) \neq 0$, this vanishes if and only if $\cMP q^{- 2 (\omega_s, 2 \rho)} = \cPM q^{(\alpha_s, \alpha_s)}$.
Together with the condition $\cPM + \cMP = 1$, this uniquely determines $\cPM$ and $\cMP$ as in the claim.
\end{proof}

We are now in the position to discuss the relation between the Ricci tensor and the quantum metric $\met$.
We recall that a smooth manifold equipped with a metric tensor $g$ is called an \emph{Einstein manifold} if we have the equality $\ricci = k g$, where the scalar $k$ is called the \emph{Einstein constant} (see for instance \cite[Section 3.1.4]{petersen}).

\begin{theorem}
\label{thm:symmetric-ricci}
Choosing the coefficients $\cPM$ and $\cMP$ as above, we have
\[
\ricci = - \frac{q^{(\alpha_s, \alpha_s)}}{1 + q^{(\alpha_s, \alpha_s)} q^{2 (\omega_s, 2 \rho)}} \qdim(V) \met = - \frac{q^2}{1 + q^{2 r + 2}} [r + 1]_q \met.
\]
Hence the quantum projective spaces, equipped with the Fubini-Study quantum metric $\met$, satisfy a quantum analogue of the Einstein condition.
\end{theorem}

\begin{proof}
Inserting the coefficients $\cPM$ and $\cMP$ in the previous expression gives
\[
\begin{split}
\ricci & = - \frac{q^{- 2 (\omega_s, 2 \rho)}}{1 + q^{-(\alpha_s, \alpha_s)} q^{-2 (\omega_s, 2 \rho)}} \qdim(V) \metPM - \frac{q^{(\alpha_s, \alpha_s)}}{1 + q^{(\alpha_s, \alpha_s)} q^{2 (\omega_s, 2 \rho)}} \qdim(V) \metMP \\
& = - \frac{q^{(\alpha_s, \alpha_s)}}{1 + q^{(\alpha_s, \alpha_s)} q^{2 (\omega_s, 2 \rho)}} \qdim(V) (\metPM + \metMP).
\end{split}
\]
Therefore the Ricci tensor is proportional to the quantum metric $\met = \metPM + \metMP$.
The second expression is obtained using $(\alpha_s, \alpha_s) = 2$, $(\omega_s, 2 \rho) = r$ and $\qdim(V) = [r + 1]_q$.
\end{proof}

\begin{remark}
As shown in \cite[Example 8.10]{quantum-book}, $\ricci$ reduces in the classical limit to the usual Ricci tensor from differential geometry, up to an overall constant.
Let us denote the latter by $\ricci_c$, as defined for instance in \cite{petersen}. Then \cite[Section 4.5.3]{petersen} gives the result $\ricci_c = 2 (r + 1) \met$ for the classical projective spaces with the Fubini-Study metric.
This should be compared with the classical limit of \cref{thm:symmetric-ricci}, namely $\ricci = - \frac{1}{2} \dim(V) \met = - \frac{1}{2} (r + 1) \met$, which shows that they coincide up to an overall constant.
\end{remark}

Finally we can look at the \emph{scalar curvature}, defined as
\[
\scal := \imet \circ \ricci.
\]

\begin{corollary}
The scalar curvature is given by
\[
\scal = - q^{-r + 1} [r]_q [r + 1]_q.
\]
\end{corollary}

\begin{proof}
This follows by combining \cref{thm:symmetric-ricci} with the quantum metric dimension from \cref{prop:quantum-trace}, that is $\imet \circ \met = (q^{r + 1} + q^{-(r + 1)}) [r]_q$. Indeed we have
\[
\scal = - \frac{q^{2}}{1 + q^{2 r + 2}} [r + 1]_q \cdot (q^{r + 1} + q^{-(r + 1)}) [r]_q  = - q^{-r + 1} [r]_q [r + 1]_q. \qedhere
\]
\end{proof}

\subsection{Discussion of the choices}

In this brief subsection we discuss the choices made for the splitting map $\splitmap$, used in the definition of the Ricci tensor.
In \cref{prop:split-map} we have restricted our attention to maps which are linear combinations of $\id$ and $\sigma$ in each component $\calc_a \otimes \calc_b$.
This led us to a one-parameter family of splitting maps, upon requiring that they should descend to $\calc^2$ and split the wedge product.
Other choices are probably available, but we should point out that these choices are severely limited if we require the splitting map to be \emph{covariant}, a natural requirement since the differential calculus $\calc$ is covariant.

The one remaining parameter in $\splitmap$ was fixed by the requirement that the Ricci tensor should be symmetric, that is $\wedge(\ricci) = 0$.
This is certainly a natural condition and immediately leads to the Einstein condition $\ricci = k \met$, as discussed above.

Here we point out a somewhat unexpected consequence of this choice.
Classically we have the splitting map $\splitmap(x \otimes y) = \frac{1}{2} (x \otimes y - y \otimes x)$ and the identity $\imet \circ \splitmap = 0$ holds, since the inverse metric is a symmetric bilinear form.
However this is not true in the quantum case with our choice.
Indeed, using \cref{prop:inverse-metric-symmetry} we compute
\[
\begin{split}
\imet \splitmap(\del p \otimes \delbar p)
& = \frac{1}{1 + q^{(\alpha_s, \alpha_s)} q^{2 (\omega_s, 2 \rho)}} (\del p, \delbar p) - \frac{1}{1 + q^{-(\alpha_s, \alpha_s)} q^{-2 (\omega_s, 2 \rho)}} \imet \sigma(\del p \otimes \delbar p) \\
& = (1 - q^{(\alpha_s, \alpha_s)} q^{2 (\omega_s, 2 \rho)}) (\del p, \delbar p).
\end{split}
\]
A similar computation leads to $\imet \splitmap(\delbar p \otimes \del p) = (1 - q^{-(\alpha_s, \alpha_s)} q^{-2 (\omega_s, 2 \rho)}) (\delbar p, \del p)$.
Therefore $\imet \circ \splitmap$ is proportional to $\imet$ (with different coefficients on the two components).

This apparent defect could be fixed by making a different choice for the splitting map.
Indeed, it is easy to check that one can also fix the free parameter in $\splitmap$ by requiring that $\splitmap \circ \imet$ should vanish.
The drawback of this choice is that the Ricci tensor acquires an antisymmetric component, that is $\wedge(\ricci)$ is not zero, and the Einstein condition does not hold.
We consider this choice to be less natural than the one made above.

\appendix

\section{Classical formulae}
\label{sec:classical-formulae}

In this appendix we write down the classical limits of the formulae derived in the paper.
The aim is to provide a better comparison between classical and quantum projective spaces.

First of all, in the classical case the braiding is simply the flip map and we have
\[
(\opS_{1 2 3} w)^{i j k l} = w^{k j i l}, \quad
(\opSt_{2 3 4} w)^{i j k l} = w^{i l k j}, \quad
(\opT_{1 2 3 4} w)^{i j k l} = w^{k l i j}.
\]
Using this fact, the relations \eqref{eq:algebra-relations} for the algebra $\subalg$ become
\[
p^{i j} p^{k l} = p^{k j} p^{i l}, \quad
p^{i j} p^{k l} = p^{i l} p^{k j}, \quad
\sum_i p^{i i} = 1.
\]
Note that combining the first two relations gives $p^{i j} p^{k l} = p^{k l} p^{i j}$, which amounts to commutativity of the generators $p^{i j}$.
For the differential calculus, the relations \eqref{eq:calculus-relations} become
\[
\begin{gathered}
p^{i j} \del p^{k l} = p^{i l} \del p^{k j}, \quad
\sum_i \del p^{i i} = 0, \\
p^{i j} \delbar p^{k l} = p^{k j} \delbar p^{i l}, \quad
\sum_i \delbar p^{i i} = 0.
\end{gathered}
\]
The quantum metric appearing in \eqref{eq:quantum-metric} reduces to
\[
g = \sum_{i, j} (\del p^{i j} \otimes \delbar p^{j i} + \delbar p^{i j} \otimes \del p^{j i}),
\]
which can be seen to correspond to the Fubini-Study metric (for more details see \cite[Appendix A]{paper-projective}).
The inverse metric from \eqref{eq:inverse-metric} becomes
\[
\begin{gathered}
(\del p^{i j}, \del p^{k l}) = 0, \quad
(\delbar p^{i j}, \delbar p^{k l}) = 0, \\
(\del p^{i j}, \delbar p^{k l}) = \delta^{i l} p^{k j} - p^{i j} p^{k l}, \quad
(\delbar p^{i j}, \del p^{k l}) = \delta^{k j} p^{i l} - p^{i j} p^{k l}.
\end{gathered}
\]

The connection \eqref{eq:connection} reduces to the \emph{Levi-Civita connection} on the cotangent bundle, defined with respect to the Fubini-Study metric.
We have the formulae
\[
\begin{split}
\nabla(\del p^{i j}) & = \sum_k \delbar p^{k j} \otimes \del p^{i k} - p^{i j} g_{- +}, \\
\nabla(\delbar p^{i j}) & = \sum_k \del p^{i k} \otimes \delbar p^{k j} - p^{i j} g_{+ -}.
\end{split}
\]
This is clearly a bimodule connection, as any connection in the commutative case.
The generalized braiding \eqref{eq:generalized-braiding} is simply the flip map.

The Riemann tensor from \cref{lem:riemann-minus} and \cref{lem:riemann-plus} reduces to
\[
\begin{split}
\curv(\del p) & = - \sum_{k, l} \del p^{l j} \wedge \delbar p^{k l} \otimes \del p^{i k} - \del p^{i j} \wedge \metMP, \\
\curv(\delbar p) & = - \sum_{k, l} \delbar p^{i k} \wedge \del p^{k l} \otimes \delbar p^{l j} - \delbar p^{i j} \wedge \metPM.
\end{split}
\]
The splitting map from \cref{prop:split-map}, taking into account \cref{thm:symmetric-ricci}, becomes
\[
\splitmap(x \wedge y) = \frac{1}{2} (x \otimes y - y \otimes x).
\]
This is the usual antisymmetrizer, corresponding to the classical splitting map.
Finally for the Ricci tensor given in \cref{thm:symmetric-ricci} and the scalar curvature we have
\[
\ricci = - \frac{1}{2} (r + 1) \met, \quad
\scal = - r (r + 1).
\]
Up to overall factors in the definition of the Riemann and Ricci tensors, these expressions correspond to the classical formulae for the Fubini-Study metric.

\section{The maps \texorpdfstring{$\opS$}{S} and \texorpdfstring{$\opSt$}{St}}
\label{sec:properties-S}

In this appendix we recall various properties satisfied by the maps
\[
\begin{split}
\opS_{1 2 3} & = (\braid_{V, V^*})_{2 3} (\braid_{V, V})_{1 2} (\braid^{-1}_{V, V^*})_{2 3}, \\
\opSt_{2 3 4} & = (\braid_{V, V^*})_{2 3} (\braid^{-1}_{V^*, V^*})_{3 4} (\braid^{-1}_{V, V^*})_{2 3}.
\end{split}
\]
Proofs of these facts can be found in \cite[Appendix A]{paper-projective}.

First some general properties, valid for any simple $\Uqg$-module $V$.

\begin{proposition}
\label{prop:S-properties}
The maps $\opS$ and $\opSt$ satisfy the following properties.

\begin{enumerate}
\item We have the commutation relations
\[
\opS_{1 2 3} \opSt_{2 3 4} = \opSt_{2 3 4} \opS_{1 2 3}, \quad
\opSt_{2 3 4} \opS_{3 4 5} = \opS_{3 4 5} \opSt_{2 3 4}.
\]
\item We have the "braid equations"
\[
\opS_{1 2 3} \opS_{3 4 5} \opS_{1 2 3} = \opS_{3 4 5} \opS_{1 2 3} \opS_{3 4 5}, \quad
\opSt_{2 3 4} \opSt_{4 5 6} \opSt_{2 3 4} = \opSt_{4 5 6} \opSt_{2 3 4} \opSt_{4 5 6}.
\]
\end{enumerate}
\end{proposition}

Now we consider the case $V = V(\omega_s)$, so that $\braid_{V, V}$ satisfies a quadratic relation as in \eqref{eq:hecke-relation}.
This quadratic relation and its analogue for $\braid_{V^*, V^*}$ lead to the following identities.

\begin{lemma}
In the quadratic case we have the relations
\begin{equation}
\label{eq:S-quadratic-relations}
\begin{split}
\opS_{1 2 3} & = q^{2 (\omega_s, \omega_s) - (\alpha_s, \alpha_s)} \opS_{1 2 3}^{-1} + (1 - q^{- (\alpha_s, \alpha_s)}) q^{(\omega_s, \omega_s)}, \\
\opSt_{2 3 4} & = q^{(\alpha_s, \alpha_s) - 2 (\omega_s, \omega_s)} \opSt_{2 3 4}^{-1} + (1 - q^{(\alpha_s, \alpha_s)}) q^{- (\omega_s, \omega_s)}.
\end{split}
\end{equation}
\end{lemma}

The use of the maps $\opS$ and $\opSt$ is the main technical difference between the presentation of $\calc$ given in \cite{paper-projective} and the original presentation from \cite{heko}.

\section{Some identities}
\label{sec:some-identities}

In this appendix we collect various identities used throughout the paper.

The first set of identities expresses certain relations between the evaluation maps $\EV$ and $\EVp$ (a proof is given in \cite[Lemma C.2]{paper-projective}).

\begin{lemma}
\label{lem:S-evaluation}
Let $V = V(\lambda)$ be a simple module. Then we have
\begin{equation}
\label{eq:S-evaluation}
\EVp_{1 2} \opS_{1 2 3} = q^{(\lambda, \lambda + 2 \rho)} \EV_{2 3}, \quad
\EVp_{3 4} \opSt_{2 3 4}^{-1} = q^{(\lambda, \lambda + 2 \rho)} \EV_{2 3}.
\end{equation}
\end{lemma}

Next, we show certain identities which hold in $\calc \otimes_\subalg \calc$, or alternatively in the tensor algebra $\talg(\calc)$.
These identities rely crucially on the fact that the tensor product is over the algebra $\subalg$, and do not hold when the tensor product is over $\bbC$.

\begin{lemma}
We have the identities
\begin{equation}
\label{eq:tensor-algebra-evaluations}
\EV_{2 3} \del p \otimes \del p = 0, \quad
\EV_{2 3} \delbar p \otimes \delbar p = 0.
\end{equation}
\end{lemma}

\begin{proof}
Using \eqref{eq:evaluations-calculi} and keeping in mind the tensor product over $\subalg$, we compute
\[
\EV_{2 3} \del p \otimes \del p = \EV_{2 3} \EV_{2 3} \del p p \otimes \del p
= \EV_{2 3} \EV_{4 5} \del p \otimes p \del p = 0.
\]
Similarly, for the second identity we compute
\[
\EV_{2 3} \delbar p \otimes \delbar p = \EV_{2 3} \EV_{4 5} \delbar p \otimes p \delbar p
= \EV_{2 3} \EV_{2 3} \delbar p p \otimes \delbar p = 0. \qedhere
\]
\end{proof}

The next result is similar in spirit to the previous one.

\begin{lemma}
We have the identities
\begin{equation}
\label{eq:tensor-algebra-identities}
\opSt_{2 3 4} \del p \otimes \del p = q^{-(\omega_s, \omega_s)} \del p \otimes \del p, \quad
\opS_{1 2 3} \delbar p \otimes \delbar p = q^{(\omega_s, \omega_s)} \delbar p \otimes \delbar p.
\end{equation}
\end{lemma}

\begin{proof}
Using \eqref{eq:evaluations-calculi} and \eqref{eq:algebra-relations} we compute
\[
\begin{split}
\opSt_{2 3 4} \del p \otimes \del p
& = \opSt_{2 3 4} \EV_{2 3} \del p p \otimes \del p = \EV_{2 3} \opSt_{4 5 6} \del p \otimes p \del p \\
& = q^{-(\omega_s, \omega_s)} \EV_{2 3} \del p \otimes p \del p = q^{-(\omega_s, \omega_s)} \del p \otimes \del p.
\end{split}
\]
Similarly, for the second identity we compute
\[
\begin{split}
\opS_{1 2 3} \delbar p \otimes \delbar p
& = \opS_{1 2 3} \EV_{4 5} \delbar p \otimes p \delbar p = \EV_{4 5} \opS_{1 2 3} \delbar p p \otimes \delbar p \\
& = q^{(\omega_s, \omega_s)} \EV_{4 5} \delbar p p \otimes \delbar p = q^{(\omega_s, \omega_s)} \delbar p \otimes \delbar p. \qedhere
\end{split}
\]
\end{proof}

Finally we derive formulae for certain combinations of the evaluation $\EVp$ and the maps $\opT$ and $\opT^{-1}$.
The following terms are zero classically, but not in the quantum case.

\begin{lemma}
\label{lem:evaluation-T}
We have the identities
\[
\begin{split}
\EVp_{1 2} \opT_{1 2 3 4}^{-1} \del p \otimes \delbar p
& = - (1 - q^{(\alpha_s, \alpha_s)}) q^{(\omega_s, 2 \rho)} \EV_{2 3} \del p \otimes \delbar p + (1 - q^{(\alpha_s, \alpha_s)}) \metPM p, \\
\EVp_{3 4} \opT_{1 2 3 4} \delbar p \otimes \del p
& = (1 - q^{(\alpha_s, \alpha_s)}) q^{(\omega_s, 2 \rho)} \EV_{2 3} \opT_{1 2 3 4} \delbar p \otimes \del p - (1 - q^{(\alpha_s, \alpha_s)}) p \metMP.
\end{split}
\]
\end{lemma}

\begin{proof}
\textit{First identity.}
Using the quadratic relation \eqref{eq:S-quadratic-relations} for $\opS_{1 2 3}^{-1}$ in $\opT_{1 2 3 4}^{-1} = \opS_{1 2 3}^{-1} \opSt_{2 3 4}^{-1}$ we get
\[
\EVp_{1 2} \opT_{1 2 3 4}^{-1} \del p \otimes \delbar p
= q^{(\alpha_s, \alpha_s) - 2 (\omega_s, \omega_s)} \EVp_{1 2} \opS_{1 2 3} \opSt_{2 3 4}^{-1} \del p \otimes \delbar p + (1 - q^{(\alpha_s, \alpha_s)}) q^{-(\omega_s, \omega_s)} \EVp_{1 2} \opSt_{2 3 4}^{-1} \del p \otimes \delbar p.
\]
We abbreviate this as $\EVp_{1 2} \opT_{1 2 3 4}^{-1} \del p \otimes \delbar p = q^{(\alpha_s, \alpha_s) - 2 (\omega_s, \omega_s)} A_1 + (1 - q^{(\alpha_s, \alpha_s)}) q^{-(\omega_s, \omega_s)} A_2$.

Consider the first term. Using \eqref{eq:S-evaluation} and the quadratic relation for $\opSt_{2 3 4}^{-1}$ we compute
\[
\begin{split}
A_1
& = \EVp_{1 2} \opS_{1 2 3} \opSt_{2 3 4}^{-1} \del p \otimes \delbar p = q^{(\omega_s, \omega_s + 2 \rho)} \EV_{2 3} \opSt_{2 3 4}^{-1} \del p \otimes \delbar p \\
& = q^{2 (\omega_s, \omega_s) - (\alpha_s, \alpha_s)} q^{(\omega_s, \omega_s + 2 \rho)} \EV_{2 3} \opSt_{2 3 4} \del p \otimes \delbar p + (1 - q^{-(\alpha_s, \alpha_s)}) q^{(\omega_s, \omega_s)} q^{(\omega_s, \omega_s + 2 \rho)} \EV_{2 3} \del p \otimes \delbar p \\
& = q^{2 (\omega_s, \omega_s) - (\alpha_s, \alpha_s)} \EVp_{3 4} \del p \otimes \delbar p + (1 - q^{-(\alpha_s, \alpha_s)}) q^{(\omega_s, \omega_s)} q^{(\omega_s, \omega_s + 2 \rho)} \EV_{2 3} \del p \otimes \delbar p \\
& = (1 - q^{-(\alpha_s, \alpha_s)}) q^{(\omega_s, \omega_s)} q^{(\omega_s, \omega_s + 2 \rho)} \EV_{2 3} \del p \otimes \delbar p.
\end{split}
\]
Now consider the second term. First we write
\[
A_2 = \EVp_{1 2} \opSt_{2 3 4}^{-1} \del p \otimes \delbar p
= \EVp_{1 2} \opSt_{2 3 4}^{-1} \EV_{2 3} \del p p \otimes \delbar p
= \EVp_{1 2} \EV_{2 3} \opSt_{4 5 6}^{-1} \del p \otimes p \delbar p.
\]
Then using the quadratic relation for $\opSt_{4 5 6}^{-1}$ and \eqref{eq:right-module-relations} we compute
\[
\begin{split}
A_2
& = q^{2 (\omega_s, \omega_s) - (\alpha_s, \alpha_s)} \EVp_{1 2} \EV_{2 3} \opSt_{4 5 6} \del p \otimes p \delbar p + (1 - q^{-(\alpha_s, \alpha_s)}) q^{(\omega_s, \omega_s)} \EVp_{1 2} \EV_{2 3} \del p \otimes p \delbar p \\
& = q^{(\omega_s, \omega_s)} \EVp_{1 2} \EV_{2 3} \del p \otimes \delbar p p
= q^{(\omega_s, \omega_s)} \metPM p.
\end{split}
\]

Using the two relations derived above we get
\[
\begin{split}
\EVp_{1 2} \opT_{1 2 3 4}^{-1} \del p \otimes \delbar p
& = q^{(\alpha_s, \alpha_s) - 2 (\omega_s, \omega_s)} A_1 + (1 - q^{(\alpha_s, \alpha_s)}) q^{-(\omega_s, \omega_s)} A_2 \\
& = - (1 - q^{(\alpha_s, \alpha_s)}) q^{(\omega_s, 2 \rho)} \EV_{2 3} \del p \otimes \delbar p + (1 - q^{(\alpha_s, \alpha_s)}) \metPM p.
\end{split}
\]

\textit{Second identity.}
Using the quadratic relation for $\opSt_{2 3 4}$ in $\opT_{1 2 3 4} = \opSt_{2 3 4} \opS_{1 2 3}$ we get
\[
\EVp_{3 4} \opT_{1 2 3 4} \delbar p \otimes \del p
= q^{(\alpha_s, \alpha_s) - 2 (\omega_s, \omega_s)} \EVp_{3 4} \opSt_{2 3 4}^{-1} \opS_{1 2 3} \delbar p \otimes \del p + (1 - q^{(\alpha_s, \alpha_s)}) q^{-(\omega_s, \omega_s)} \EVp_{3 4} \opS_{1 2 3} \delbar p \otimes \del p.
\]
We abbreviate this as $\EVp_{3 4} \opT_{1 2 3 4} \delbar p \otimes \del p = q^{(\alpha_s, \alpha_s) - 2 (\omega_s, \omega_s)} B_1 + (1 - q^{(\alpha_s, \alpha_s)}) q^{-(\omega_s, \omega_s)} B_2$.

Consider the first term. Using \eqref{eq:S-evaluation} and the quadratic relation for $\opS_{1 2 3}$ we compute
\[
\begin{split}
B_1
& = \EVp_{3 4} \opSt_{2 3 4}^{-1} \opS_{1 2 3} \delbar p \otimes \del p
= q^{(\omega_s, \omega_s + 2 \rho)} \EV_{2 3} \opS_{1 2 3} \delbar p \otimes \del p \\
& = q^{2 (\omega_s, \omega_s) - (\alpha_s, \alpha_s)} q^{(\omega_s, \omega_s + 2 \rho)} \EV_{2 3} \opS^{-1}_{1 2 3} \delbar p \otimes \del p
+ (1 - q^{-(\alpha_s, \alpha_s)}) q^{(\omega_s, \omega_s)} q^{(\omega_s, \omega_s + 2 \rho)} \EV_{2 3} \delbar p \otimes \del p \\
& = (1 - q^{-(\alpha_s, \alpha_s)}) q^{(\omega_s, \omega_s)} q^{(\omega_s, \omega_s + 2 \rho)} \EV_{2 3} \delbar p \otimes \del p.
\end{split}
\]
In \cite[Lemma C.5]{paper-projective} it is shown that $p \metMP = q^{(\omega_s, 2 \rho)} \EV_{2 3} \delbar p \otimes \del p$. Hence we get
\[
B_1 = (1 - q^{-(\alpha_s, \alpha_s)}) q^{2 (\omega_s, \omega_s)} p \metMP.
\]
On the other hand, we rewrite the second term using \eqref{eq:S-evaluation} as
\[
\begin{split}
B_2 & = \EVp_{3 4} \opS_{1 2 3} \delbar p \otimes \del p
= q^{(\omega_s, \omega_s + 2 \rho)} \EV_{2 3} \opSt_{2 3 4} \opS_{1 2 3} \delbar p \otimes \del p \\
& = q^{(\omega_s, \omega_s + 2 \rho)} \EV_{2 3} \opT_{1 2 3 4} \delbar p \otimes \del p.
\end{split}
\]

Using these expressions we obtain
\[
\begin{split}
\EVp_{3 4} \opT_{1 2 3 4} \delbar p \otimes \del p
& = q^{(\alpha_s, \alpha_s) - 2 (\omega_s, \omega_s)} B_1 + (1 - q^{(\alpha_s, \alpha_s)}) q^{-(\omega_s, \omega_s)} B_2 \\
& = - (1 - q^{(\alpha_s, \alpha_s)}) p \metMP + (1 - q^{(\alpha_s, \alpha_s)}) q^{(\omega_s, 2 \rho)} \EV_{2 3} \opT_{1 2 3 4} \delbar p \otimes \del p. \qedhere
\end{split}
\]
\end{proof}

\section{Relations in degree two}
\label{sec:relations-degree-two}

In this appendix we derive alternative forms for the mixed relations between the generators $\del p$ and $\delbar p$ in the Heckenberger-Kolb calculus $\calc^\bullet$.
We should point out that the expressions we derive here are valid for any irreducible flag manifold, as in the setting of \cite{heko}.

\begin{lemma}
\label{lem:degree-two-relations}
We have the relations
\[
\begin{split}
\del p \wedge \delbar p & = - q^{(\alpha_s, \alpha_s)} \opT_{1 2 3 4} \delbar p \wedge \del p + q^{(\alpha_s, \alpha_s)} q^{-(\omega_s, 2 \rho)} \EVp_{3 4} \opT_{1 2 3 4} \delbar p \wedge \del p p, \\
\delbar p \wedge \del p & = - q^{-(\alpha_s, \alpha_s)} \opT_{1 2 3 4} \del p \wedge \delbar p + q^{-(\alpha_s, \alpha_s)} q^{-(\omega_s, 2 \rho)} \EVp_{3 4} \opT_{1 2 3 4} \del p \wedge \delbar p p.
\end{split}
\]
We also have the relations
\[
\begin{split}
\del p \wedge \delbar p & = - q^{(\alpha_s, \alpha_s)} \opT_{1 2 3 4}^{-1} \delbar p \wedge \del p + q^{(\alpha_s, \alpha_s)} q^{-(\omega_s, 2 \rho)} \EVp_{3 4} \opT_{3 4 5 6}^{-1} p \delbar p \wedge \del p, \\
\delbar p \wedge \del p & = - q^{-(\alpha_s, \alpha_s)} \opT_{1 2 3 4}^{-1} \del p \wedge \delbar p + q^{-(\alpha_s, \alpha_s)} q^{-(\omega_s, 2 \rho)} \EVp_{3 4} \opT_{3 4 5 6}^{-1} p \del p \wedge \delbar p.
\end{split}
\]
\end{lemma}

\begin{proof}
Consider the relation $\del p p = q^{(\alpha_s, \alpha_s)} \opT_{1 2 3 4} p \del p$ from \eqref{eq:right-module-relations}. Applying $\delbar$ we get
\[
\delbar \del p p - \del p \wedge \delbar p = q^{(\alpha_s, \alpha_s)} \opT_{1 2 3 4} \delbar p \wedge \del p + q^{(\alpha_s, \alpha_s)} \opT_{1 2 3 4} p \delbar \del p.
\]
We have the identity $\opT_{1 2 3 4} p \del \delbar p = \del \delbar p p$, as shown for instance in \cite[Lemma C.7]{paper-projective}.
Then the relation above can be rewritten as
\[
\del p \wedge \delbar p = - q^{(\alpha_s, \alpha_s)} \opT_{1 2 3 4} \delbar p \wedge \del p + (1 - q^{(\alpha_s, \alpha_s)}) \delbar \del p p.
\]
Applying $\EVp_{3 4}$ to this we get $(1 - q^{(\alpha_s, \alpha_s)}) q^{(\omega_s, 2 \rho)} \delbar \del p = q^{(\alpha_s, \alpha_s)} \EVp_{3 4} \opT_{1 2 3 4} \delbar p \wedge \del p$, since $\EVp_{1 2} \delbar p = 0$ and $\EVp_{1 2} p = q^{(\omega_s, 2 \rho)}$.
Plugging this back into $\del p \wedge \delbar p$ gives
\[
\del p \wedge \delbar p = - q^{(\alpha_s, \alpha_s)} \opT_{1 2 3 4} \delbar p \wedge \del p + q^{(\alpha_s, \alpha_s)} q^{-(\omega_s, 2 \rho)} \EVp_{3 4} \opT_{1 2 3 4} \delbar p \wedge \del p p.
\]

Similarly, the relation above (before applying $\EVp_{3 4}$) can be rewritten as
\[
\delbar p \wedge \del p = - q^{-(\alpha_s, \alpha_s)} \opT_{1 2 3 4}^{-1} \del p \wedge \delbar p - (1 - q^{-(\alpha_s, \alpha_s)}) \opT_{1 2 3 4}^{-1} \delbar \del p p.
\]
Using again $\opT_{1 2 3 4} p \del \delbar p = \del \delbar p p$ gives
\[
\delbar p \wedge \del p = - q^{-(\alpha_s, \alpha_s)} \opT_{1 2 3 4}^{-1} \del p \wedge \delbar p - (1 - q^{-(\alpha_s, \alpha_s)}) p \delbar \del p.
\]
Now applying $\EVp_{1 2}$ we get $- (1 - q^{-(\alpha_s, \alpha_s)}) q^{(\omega_s, 2 \rho)} \delbar \del p = q^{-(\alpha_s, \alpha_s)} \EVp_{1 2} \opT_{1 2 3 4}^{-1} \del p \wedge \delbar p$. Hence
\[
\delbar p \wedge \del p = - q^{-(\alpha_s, \alpha_s)} \opT_{1 2 3 4}^{-1} \del p \wedge \delbar p + q^{-(\alpha_s, \alpha_s)} q^{-(\omega_s, 2 \rho)} \EVp_{3 4} \opT_{3 4 5 6}^{-1} p \del p \wedge \delbar p.
\]

The other two relations can be obtained in a similar way starting from $\delbar p p = q^{-(\alpha_s, \alpha_s)} \opT_{1 2 3 4} p \delbar p$, again from \eqref{eq:right-module-relations}. Applying $\del$ we get
\[
\del \delbar p p - \delbar p \wedge \del p = q^{-(\alpha_s, \alpha_s)} \opT_{1 2 3 4} \del p \wedge \delbar p + q^{-(\alpha_s, \alpha_s)} \opT_{1 2 3 4} p \del \delbar p.
\]
Taking into account $\opT_{1 2 3 4} p \del \delbar p = \del \delbar p p$ we rewrite this as
\[
\delbar p \wedge \del p = - q^{-(\alpha_s, \alpha_s)} \opT_{1 2 3 4} \del p \wedge \delbar p + (1 - q^{-(\alpha_s, \alpha_s)}) \del \delbar p p.
\]
Applying $\EVp_{3 4}$ we get $(1 - q^{-(\alpha_s, \alpha_s)}) q^{(\omega_s, 2 \rho)} \del \delbar p = q^{-(\alpha_s, \alpha_s)} \EVp_{3 4} \opT_{1 2 3 4} \del p \wedge \delbar p$. Then
\[
\delbar p \wedge \del p = - q^{-(\alpha_s, \alpha_s)} \opT_{1 2 3 4} \del p \wedge \delbar p + q^{-(\alpha_s, \alpha_s)} q^{-(\omega_s, 2 \rho)} \EVp_{3 4} \opT_{1 2 3 4} \del p \wedge \delbar p p.
\]

Similarly we can rewrite the expression above as
\[
\del p \wedge \delbar p = - q^{(\alpha_s, \alpha_s)} \opT_{1 2 3 4}^{-1} \delbar p \wedge \del p - (1 - q^{(\alpha_s, \alpha_s)}) p \del \delbar p.
\]
Applying $\EVp_{1 2}$ we get $- (1 - q^{(\alpha_s, \alpha_s)}) q^{(\omega_s, 2 \rho)} \del \delbar p = q^{(\alpha_s, \alpha_s)} \EVp_{1 2} \opT_{1 2 3 4}^{-1} \delbar p \wedge \del p$. Then
\[
\del p \wedge \delbar p = - q^{(\alpha_s, \alpha_s)} \opT_{1 2 3 4}^{-1} \delbar p \wedge \del p + q^{(\alpha_s, \alpha_s)} q^{-(\omega_s, 2 \rho)} \EVp_{3 4} \opT_{3 4 5 6}^{-1} p \delbar p \wedge \del p. \qedhere
\]
\end{proof}

\begin{remark}
The relation derived in \cite[Proposition 3.11]{heko} and recalled in \eqref{eq:universal-mixed} corresponds to the fourth one of the previous lemma.
Clearly any of these four expressions can be taken as the defining relation for the mixed terms in degree two.
\end{remark}

\end{document}